\documentclass[11pt]{article}
\usepackage[numbers,square]{natbib}
\usepackage{amsmath}
\usepackage{amsthm}
\usepackage{amsfonts}
\usepackage{amssymb}
\usepackage{siunitx}
\usepackage{commath}
\usepackage{graphicx}
\usepackage{xspace}
\usepackage{color}
\usepackage[lofdepth,lotdepth]{subfig}
\usepackage{stmaryrd}
\usepackage{url}
\usepackage{amsthm}
\usepackage{graphbox}
\usepackage{footnote}
\makesavenoteenv{tabular}
\makesavenoteenv{table}
\usepackage[hidelinks]{hyperref}
\hypersetup{colorlinks = true, urlcolor = blue,
  linkcolor = blue, citecolor = blue}
\usepackage{cleveref}

\usepackage[margin=0.7in]{geometry}
\makeatletter
\newcommand{\tnorm}{\@ifstar\@tnorms\@tnorm}
\newcommand{\@tnorms}[1]{%
  \left|\mkern-1.5mu\left|\mkern-1.5mu\left|
   #1
  \right|\mkern-1.5mu\right|\mkern-1.5mu\right|
}
\newcommand{\@tnorm}[2][]{%
  \mathopen{#1|\mkern-1.5mu#1|\mkern-1.5mu#1|}
  #2
  \mathclose{#1|\mkern-1.5mu#1|\mkern-1.5mu#1|}
}
\makeatother

\newcommand{\jump}[1]{\llbracket #1 \rrbracket}

\newtheorem{theorem}{Theorem}
\newtheorem{lemma}{Lemma}

\newtheorem{remark}{Remark}
\newtheorem{definition}{Definition}
\title{Parameter-robust preconditioning for hybridizable symmetric
  discretizations}
\author{
  E. Henr\'iquez\thanks{Department of Applied Mathematics, University of
    Waterloo, ON, Canada (\url{ehenriqu@uwaterloo.ca}),
    \url{http://orcid.org/0000-0002-0243-0368}}
  \and
  J. J. Lee\thanks{Department of Mathematics, Baylor University,
    TX, USA (\url{jeonghun_lee@baylor.edu}),
    \url{https://orcid.org/0000-0001-5201-8526}}
  \and
  S. Rhebergen\thanks{Department of Applied Mathematics, University of
    Waterloo, ON, Canada (\url{srheberg@uwaterloo.ca}),
    \url{http://orcid.org/0000-0001-6036-0356}}}
\begin{document}
\maketitle
\begin{abstract}
  Hybridizable discretizations allow for the elimination of local
  degrees-of-freedom leading to reduced linear systems. In this paper,
  we determine and analyse an approach to construct parameter-robust
  preconditioners for these reduced systems. Using the framework of
  Mardal and Winther (Numer. Linear Algebra Appl., 18(1):1--40, 2011)
  we first determine a parameter-robust preconditioner for the full
  system. We then eliminate the local degrees-of-freedom of this
  preconditioner to obtain a preconditioner for the reduced
  system. However, not all reduced preconditioners obtained in this
  way are automatically robust. We therefore present conditions that
  must be satisfied for the reduced preconditioner to be robust. To
  demonstrate our approach, we determine preconditioners for the
  reduced systems obtained from hybridizable discretizations of the
  Darcy and Stokes equations. Our analysis is verified by numerical
  examples in two and three dimensions.
\end{abstract}
\section{Introduction}
\label{s:introduction}
Hybridizable discretizations of partial differential equations (PDEs)
allow local degrees-of-freedom to be eliminated from the
discretization. These discretizations have appeared, for example, in
the context of continuous Galerkin methods
\cite{karniadakis2005spectral,kirby2012cg}, hybrid mixed methods
\cite{arnold1985mixed,fraeijs1965displacement}, hybridizable
discontinuous Galerkin (HDG) methods and their variants
\cite{cockburn2009unified,soon2009hybridizable}, Weak Galerkin (WG)
\cite{wang2013weak}, and hybrid high-order methods (HHO)
\cite{di2015hybrid,di2015hybrid2}. In terms of solving reduced
problems, various preconditioners have been studied for hybridizable
discretizations, for example, for hybrized mixed methods
\cite{cowsar1995balancing,dobrev2019algebraic,gopalakrishnan2003schwarz,gopalakrishnan2009convergent},
hybridizable discontinuous Galerkin (HDG) methods
\cite{cockburn2014multigrid,fu2021uniform,gander2015analysis,muralikrishnan2020multilevel},
and hybrid high-order methods (HHO)
\cite{botti2022p,di2021h,di2024homogeneous}.

Many PDEs are parameter-dependent and these model parameters may have
a large effect on the solution of the discrete system by a
preconditioned Krylov space method if they are not taken into account
by the preconditioner. A simple approach to construct preconditioners
for discretizations of parameter-dependent systems was proposed by
Mardal and Winther \cite{mardal2011preconditioning} (see also
\cite{hiptmair2006}). This approach constructs preconditioners in a
Hilbert space setting as follows. Let $A:X \to X^{\ast}$ be a bounded
linear operator, where $X$ is a Hilbert space and $X^{\ast}$ its
dual. The preconditioner $P^{-1}$ is designed to be the canonical
mapping $P^{-1}:X^{\ast} \to X$ such that
$\langle P \cdot, \cdot \rangle_{X^{\ast},X}$, with
$\langle \cdot, \cdot \rangle_{X^{\ast},X}$ the duality pairing
between $X^{\ast}$ and $X$, is an inner product on $X$. If $A$, in
addition, satisfies an inf-sup condition and is bounded in the norm
induced by the inner product
$\langle P \cdot, \cdot \rangle_{X^{\ast},X}$, with the inf-sup
stability and boundedness constants independent of the model and
discretization parameters, then $P$ is a parameter-robust
preconditioner. More details are presented in
\cref{ss:preconframework}. Parameter-robust preconditioners using this
framework for hybridizable discretizations have been studied, for
example, in \cite{fu2023uniform} for the Stokes and elasticity
equations, and \cite{kraus2021uniformly} for Biot's consolidation
model.

In this paper, given a parameter-robust preconditioner for a
hybridizable symmetric discretization of a parameter-dependent PDE, we
determine conditions that must be satisfied for a reduced form of this
preconditioner to also be parameter-robust for the reduced form of a
hybridizable discretization (see \cref{thm:SpparamrobSa}). In
\cref{s:application} we demonstrate this approach by determining and
analyzing preconditioners for a hybridized mixed method for the Darcy
equation and for an HDG method for the Stokes equations. Furthermore,
to demonstrate that the conditions of \cref{thm:SpparamrobSa} are
necessary, we present an example of a preconditioner that is robust
for the Darcy problem before static condensation, but not robust after
static condensation. Numerical experiments in \cref{s:num} verify the
analysis while conclusions are drawn in \cref{s:conclusions}.

\section{Parameter-robust preconditioning}
\label{s:ParamPrecond}

After presenting the Mardal--Winther preconditioning framework in
\cref{ss:preconframework}, we present in \cref{ss:afterHyb} conditions
that must be satisfied for a preconditioner to be parameter-robust for
the reduced system after static condensation.

\subsection{Mardal--Winther preconditioning framework}
\label{ss:preconframework}
We briefly summarize the preconditioning framework presented in
\cite{mardal2011preconditioning}. Let $\boldsymbol{X}_h$ be a family
of finite element spaces defined on meshes of mesh size $h$,
$(\cdot, \cdot)_{\boldsymbol{X}_h}$ be a mesh-dependent inner product
on $\boldsymbol{X}_h$, $\norm[0]{\cdot}_{\boldsymbol{X}_h}$ the norm
induced by $(\cdot, \cdot)_{\boldsymbol{X}_h}$, and let
$\boldsymbol{X}_h^{\ast}$ be the dual space of $\boldsymbol{X}_h$. We
denote the duality pairing between $\boldsymbol{X}_h^{\ast}$ and
$\boldsymbol{X}_h$ by
$\langle \cdot, \cdot
\rangle_{\boldsymbol{X}_h^{\ast},\boldsymbol{X}_h}$. Similar notation
is used for finite element spaces other than $\boldsymbol{X}_h$.

Assume $a_h(\cdot, \cdot)$ is a symmetric, $h$-dependent, bilinear
form on $\boldsymbol{X}_h \times \boldsymbol{X}_h$ and let
$\boldsymbol{f}_h \in \boldsymbol{X}_h^{\ast}$. The discrete problem
given by: Find $\boldsymbol{x}_h \in \boldsymbol{X}_h$ such that
\begin{equation}
  \label{eq:wfgen}
  a_h(\boldsymbol{x}_h,\boldsymbol{y}_h) = \langle \boldsymbol{f}_h, \boldsymbol{y}_h \rangle_{\boldsymbol{X}_h^{\ast},\boldsymbol{X}_h}
  \quad \forall \boldsymbol{y}_h \in \boldsymbol{X}_h,
\end{equation}
is well-posed if the following conditions are satisfied:
\begin{subequations}
  \label{eq:wpconditions}
  \begin{align}
    \label{eq:wpconditions_b}
    a_h(\boldsymbol{x}_h, \boldsymbol{y}_h) &\le c_b \norm[0]{\boldsymbol{x}_h}_{\boldsymbol{X}_h}\norm[0]{\boldsymbol{y}_h}_{\boldsymbol{X}_h} && \forall \boldsymbol{x}_h, \boldsymbol{y}_h \in \boldsymbol{X}_h,
    \\
    \label{eq:wpconditions_i}
    c_i &\le \inf_{\boldsymbol{x}_h \in \boldsymbol{X}_h} \sup_{\boldsymbol{y}_h \in \boldsymbol{X}_h} \frac{a_h(\boldsymbol{x}_h, \boldsymbol{y}_h)}{\norm[0]{\boldsymbol{x}_h}_{\boldsymbol{X}_h}\norm[0]{\boldsymbol{y}_h}_{\boldsymbol{X}_h}},
  \end{align}
\end{subequations}
where $c_b>0$ and $c_i>0$ are the boundedness and stability constants.

Let $A:\boldsymbol{X}_h \to \boldsymbol{X}_h^{\ast}$ be defined by
$\langle A\boldsymbol{x}_h,\boldsymbol{y}_h
\rangle_{\boldsymbol{X}_h^{\ast},\boldsymbol{X}_h} =
a_h(\boldsymbol{x}_h,\boldsymbol{y}_h)$ for all
$\boldsymbol{x}_h,\boldsymbol{y}_h \in \boldsymbol{X}_h$. The
following problem is then equivalent to \cref{eq:wfgen}: Find
$\boldsymbol{x}_h \in \boldsymbol{X}_h$ such that
\begin{equation}
  \label{eq:wfgenop}
  A\boldsymbol{x}_h = \boldsymbol{f}_h \text{ in } \boldsymbol{X}_h^{\ast}.
\end{equation}
Krylov space methods are often used to solve
\cref{eq:wfgenop}. However, it is known that if
$\boldsymbol{X}_h^{\ast} \ne \boldsymbol{X}_h$ then a preconditioner
is required for Krylov space methods to be well-defined
\cite{mardal2011preconditioning}. If the preconditioner
$P^{-1}:\boldsymbol{X}_h^{\ast} \to \boldsymbol{X}_h$ is defined by
\begin{equation}
  \label{eq:defOpP}
  (P^{-1}\boldsymbol{f}_h, \boldsymbol{y}_h)_{\boldsymbol{X}_h}
  = \langle \boldsymbol{f}_h, \boldsymbol{y}_h \rangle_{\boldsymbol{X}_h^{\ast},\boldsymbol{X}_h}
  \quad \forall \boldsymbol{y}_h \in \boldsymbol{X}_h,\ \boldsymbol{f}_h \in \boldsymbol{X}_h^{\ast},
\end{equation}
then it can be shown that $\kappa(P^{-1}A)$, the condition number of
$P^{-1}A$, satisfies
\begin{equation*}
  \kappa(P^{-1}A) = \norm[0]{P^{-1}A}_{\mathcal{L}(\boldsymbol{X}_h,\boldsymbol{X}_h)}
  \norm[0]{(P^{-1}A)^{-1}}_{\mathcal{L}(\boldsymbol{X}_h,\boldsymbol{X}_h)} \le c_a/c_i.
\end{equation*}
Here $\mathcal{L}(\boldsymbol{X}_h,\boldsymbol{X}_h)$ is the set of
bounded linear operators mapping $\boldsymbol{X}_h$ to
$\boldsymbol{X}_h$, and we have the following expressions for the
operator norms:
\begin{align*}
  \norm[0]{P^{-1}A}_{\mathcal{L}(\boldsymbol{X}_h,\boldsymbol{X}_h)}
  &= \sup_{\boldsymbol{x}_h,\boldsymbol{y}_h \in \boldsymbol{X}_h}
    \frac{ (P^{-1}A\boldsymbol{x}_h, \boldsymbol{y}_h)_{\boldsymbol{X}_h}}{\norm[0]{\boldsymbol{x}_h}_{\boldsymbol{X}_h}\norm[0]{\boldsymbol{y}_h}_{\boldsymbol{X}_h}},
  \\
  \norm[0]{(P^{-1}A)^{-1}}_{\mathcal{L}(\boldsymbol{X}_h,\boldsymbol{X}_h)}^{-1}
  &= \inf_{\boldsymbol{x}_h \in \boldsymbol{X}_h}\sup_{\boldsymbol{y}_h \in \boldsymbol{X}_h}
    \frac{ (P^{-1}A\boldsymbol{x}_h, \boldsymbol{y}_h)_{\boldsymbol{X}_h}}{\norm[0]{\boldsymbol{x}_h}_{\boldsymbol{X}_h}\norm[0]{\boldsymbol{y}_h}_{\boldsymbol{X}_h}}.
\end{align*}

\begin{remark}
  \label{rem:Phparamrobust}
  The convergence rate of Krylov space methods for symmetric systems
  (e.g., MINRES), depends on the condition number. Therefore, if the
  constants $c_a,c_i$ in \cref{eq:wpconditions} are independent of $h$
  and other model parameters, then $P^{-1}$ is a parameter-robust
  preconditioner.
\end{remark}

In what follows, we will refer to a constant that is independent of
the discretization and model parameters as a uniform constant.

\begin{remark}
  \label{rem:Phinnerprod}
  From \cref{eq:defOpP} we note that
  $\langle P \boldsymbol{x}_h, \boldsymbol{y}_h
  \rangle_{\boldsymbol{X}_h^{\ast},\boldsymbol{X}_h} =
  (\boldsymbol{x}_h, \boldsymbol{y}_h)_{\boldsymbol{X}_h}$. In other
  words, $P:\boldsymbol{X}_h \to \boldsymbol{X}_h^{\ast}$ is the
  operator associated with the inner product on
  $\boldsymbol{X}_h$. Furthermore, note that
  $\langle P \boldsymbol{x}_h, \boldsymbol{x}_h
  \rangle_{\boldsymbol{X}_h^{\ast},\boldsymbol{X}_h} =
  \norm[0]{\boldsymbol{x}_h}_{\boldsymbol{X}_h}^2$ for all
  $\boldsymbol{x}_h \in \boldsymbol{X}_h$. In practice, we may replace
  $P$ by a norm-equivalent operator $\widehat{P}$ such that
  $c_0 \norm[0]{\boldsymbol{x}_h}_{\boldsymbol{X}_h}^2 \le \langle
  \widehat{P}\boldsymbol{x}_h, \boldsymbol{x}_h
  \rangle_{\boldsymbol{X}_h^{\ast},\boldsymbol{X}_h} \le c_1
  \norm[0]{\boldsymbol{x}_h}_{\boldsymbol{X}_h}^2$, with $c_0,c_1$
  positive uniform constants.
\end{remark}

\subsection{Preconditioning for a hybridizable discretization}
\label{ss:afterHyb}

The Mardal--Winther framework presents an approach to designing
parameter-robust preconditioners for discretizations of PDEs. This
framework relies on identifying a norm
$\norm[0]{\cdot}_{\boldsymbol{X}_h}$ for the finite element space
$\boldsymbol{X}_h$ such that the boundedness and stability constants
in \cref{eq:wpconditions} are uniform constants. In this section, we
extend this framework to hybridizable discretizations. In particular,
given a block preconditioner $P^{-1}$ for the block operator
$A:\boldsymbol{X}_h \to \boldsymbol{X}_h^{\ast}$, we identify
conditions that the Schur complement of $P$ must satisfy to be a
parameter-robust preconditioner for the Schur complement of $A$.

Let $\boldsymbol{X}_h := X_h \times \bar{X}_h$ and assume that
\cref{eq:wfgenop} can be written as
\begin{equation}
  \label{eq:wfgenop_block}
  \begin{bmatrix}
    A_{11} & A_{21}^T
    \\
    A_{21} & A_{22}
  \end{bmatrix}
  \begin{bmatrix}
    x_h \\ \bar{x}_h
  \end{bmatrix}
  =
  \begin{bmatrix}
    f_h \\ \bar{f}_h
  \end{bmatrix},
\end{equation}
where $\boldsymbol{x}_h = (x_h,\bar{x}_h) \in X_h \times \bar{X}_h$,
and $A_{11}:X_h \to X_h^{\ast}$, $A_{21}:X_h \to \bar{X}_h^{\ast}$,
and $A_{22}:\bar{X}_h \to \bar{X}_h^{\ast}$ are the operators that
form $A:\boldsymbol{X}_h \to \boldsymbol{X}_h^{\ast}$. Assuming that
it is cheap to eliminate $x_h$, we obtain the following equation for
$\bar{x}_h$:
\begin{equation}
  \label{eq:SCop}
  S_A \bar{x}_h = \bar{b}_h,
  \quad \text{where} \quad
  S_A = A_{22} - A_{21}A_{11}^{-1}A_{21}^T,
  \quad
  \bar{b}_h = \bar{f}_h - A_{21}A_{11}^{-1}f_h.
\end{equation}
Next, let $P_{11}:X_h \to X_h^{\ast}$,
$P_{21}:X_h \to \bar{X}_h^{\ast}$, and
$P_{22}:\bar{X}_h \to \bar{X}_h^{\ast}$ be operators such that
\begin{equation}
  \label{eq:Pblockstructure}
  P =
  \begin{bmatrix}
    P_{11} & P_{21}^T
    \\
    P_{21} & P_{22}
  \end{bmatrix},
\end{equation}
with $P$ defined by \cref{eq:defOpP}. Since $P$ is a positive
operator, in the sense that $P$ is symmetric and
$\langle P\boldsymbol{x}_h,\boldsymbol{x}_h
\rangle_{\boldsymbol{X}_h^{\ast},\boldsymbol{X}_h} > 0$
$\forall \boldsymbol{x}_h \in
\boldsymbol{X}_h\backslash\cbr[0]{\boldsymbol{0}}$, and if $P_{11}$ is
a positive operator ($P_{11}$ is symmetric and
$\langle P_{11}x_h, x_h \rangle_{X_h^{\ast},X_h} > 0$
$\forall x_h \in X_h\backslash\cbr[0]{0}$), then so is the operator
\begin{equation}
  \label{eq:defSCSP}
  S_P := P_{22} - P_{21}P_{11}^{-1}P_{21}^T.
\end{equation}
This follows by considering the matrix representation of the operators
in \cref{eq:Pblockstructure} and \cite[Chapter A.5.5]{Boyd:book}. The
operator $S_{P} : \bar{X}_h \to \bar{X}_h^{\ast}$ can therefore be
used to define an inner product on $\bar{X}_h$ in the following sense:
\begin{equation*}
  \langle S_{P} \bar{x}_h, \bar{y}_h \rangle_{\bar{X}_h^{\ast},\bar{X}_h}
  = (\bar{x}_h, \bar{y}_h)_{\bar{X}_h}
  \qquad \forall \bar{x}_h, \bar{y}_h \in \bar{X}_h.
\end{equation*}
Analogously, we note that
$S_{P}^{-1} : \bar{X}_h^{\ast} \to \bar{X}_h$ is such that
\begin{equation*}
  (S_{P}^{-1}\bar{f}_h, \bar{y}_h)_{\bar{X}_h} = \langle \bar{f}_f, \bar{y}_h \rangle_{\bar{X}_h^*, \bar{X}_h}
  \qquad \forall \bar{f}_h \in \bar{X}_h^{\ast}, \ \bar{y}_h \in \bar{X}_h.
\end{equation*}
The following theorem presents conditions for $S_{P}^{-1}$ to be a
parameter-robust preconditioner for $S_A$.

\begin{theorem}
  \label{thm:SpparamrobSa}
  Let $A$ be the operator with the structure given in
  \cref{eq:wfgenop_block} and let $P$ be the operator satisfying
  \cref{eq:defOpP} with structure given in
  \cref{eq:Pblockstructure}. Furthermore, let $S_A$ and $S_P$ be the
  operators defined in \cref{eq:SCop} and \cref{eq:defSCSP},
  respectively. Assume that $A$ satisfies \cref{eq:wpconditions},
  $A_{11}$ is invertible, and $P_{11}$ is a positive operator. If
  there exists a uniform constant $c_l > 0$ such that
  \begin{equation}
    \label{eq:liftingcondition}
    \norm[0]{(-A_{11}^{-1}A_{21}^T\bar{x}_h,\bar{x}_h)}_{\boldsymbol{X}_h} \le c_l \norm[0]{\bar{x}_h}_{\bar{X}_h}
    \quad \forall \bar{x}_h \in \bar{X}_h,
  \end{equation}
  then
  \begin{equation}
    \label{eq:spectral-equivalence}
    \norm[0]{ S_{P}^{-1} S_{A} }_{\bar{X}_h},\quad \norm[0]{ (S_{P}^{-1} S_{A})^{-1} }_{\bar{X}_h} \quad \text{ are uniformly bounded.}
  \end{equation}
  Moreover, if \cref{eq:spectral-equivalence} holds, then
  \cref{eq:liftingcondition} holds with a uniform constant $c_l$.
\end{theorem}
\begin{proof}
  \textbf{Step 1.} We first prove that \cref{eq:liftingcondition}
  implies \cref{eq:spectral-equivalence}. From
  \cref{eq:wpconditions_b}, we have that
  \begin{align*}
    c_b
    & \ge
      \sup_{\substack{\boldsymbol{x}_h \in \boldsymbol{X}_h
    \\
    \boldsymbol{y}_h \in \boldsymbol{X}_h}}
    \frac{\langle A\boldsymbol{x}_h,\boldsymbol{y}_h\rangle_{\boldsymbol{X}_h^{\ast},\boldsymbol{X}_h}}{\norm[0]{\boldsymbol{x}_h}_{\boldsymbol{X}_h}\norm[0]{\boldsymbol{y}_h}_{\boldsymbol{X}_h}}
    \\
    &=
      \sup_{\substack{\boldsymbol{x}_h \in \boldsymbol{X}_h
    \\
    \boldsymbol{y}_h \in \boldsymbol{X}_h}}
    \frac{\langle A_{11}(x_h + A_{11}^{-1}A_{21}^T\bar{x}_h), y_h + A_{11}^{-1}A_{21}^{T}\bar{y}_h \rangle_{X_h^{\ast},X_h}
    + \langle S_A\bar{x}_h, \bar{y}_h \rangle_{\bar{X}_h^{\ast},\bar{X}_h}}{\norm[0]{\boldsymbol{x}_h}_{\boldsymbol{X}_h}\norm[0]{\boldsymbol{y}_h}_{\boldsymbol{X}_h}}.
  \end{align*}
  As we are taking the supremum, we can bound below by choosing
  $x_h = - A_{11}^{-1}A_{21}^T \bar{x}_h$ and
  $y_h = - A_{11}^{-1}A_{21}^T \bar{y}_h$. We obtain:
  \begin{equation*}
    c_b \ge
    \sup_{\bar{x}_h,\bar{y}_h \in \bar{X}_h}
    \frac{\langle S_A\bar{x}_h, \bar{y}_h \rangle_{\bar{X}_h^{\ast},\bar{X}_h}}
    {\norm[0]{(-A_{11}^{-1}A_{21}^T\bar{x}_h, \bar{x}_h)}_{\boldsymbol{X}_h}\norm[0]{(-A_{11}^{-1}A_{21}^T\bar{y}_h, \bar{y}_h)}_{\boldsymbol{X}_h}}.
  \end{equation*}
  Combined with \cref{eq:liftingcondition},
  \begin{equation*}
      c_l^2c_b
      \ge
      \sup_{\bar{x}_h,\bar{y}_h \in \bar{X}_h}
      \frac{\langle S_A\bar{x}_h, \bar{y}_h \rangle_{\bar{X}_h^{\ast},\bar{X}_h}}
      {\norm[0]{\bar{x}_h}_{\bar{X}_h}\norm[0]{\bar{y}_h}_{\bar{X}_h}}
      =
      \sup_{\bar{x}_h,\bar{y}_h \in \bar{X}_h}
      \frac{(S_P^{-1} S_A\bar{x}_h, \bar{y}_h)_{\bar{X}_h}}
      {\norm[0]{\bar{x}_h}_{\bar{X}_h}\norm[0]{\bar{y}_h}_{\bar{X}_h}}
      =
      \norm[0]{S_P^{-1}S_A}_{\mathcal{L}(\bar{X}_h,\bar{X}_h)}.      
  \end{equation*}
  Next, observe that
  \begin{multline}
    \label{eq:AuhvhgeSA}
    \langle A\boldsymbol{x}_h, \boldsymbol{y}_h\rangle_{\boldsymbol{X}_h^{\ast},\boldsymbol{X}_h}
    \\
    =
    \langle A_{11}(x_h + A_{11}^{-1}A_{21}^T\bar{x}_h), y_h + A_{11}^{-1}A_{21}^{T}\bar{y}_h \rangle_{X_h^{\ast},X_h}
    + \langle S_A\bar{x}_h, \bar{y}_h \rangle_{\bar{X}_h^{\ast},\bar{X}_h}.
  \end{multline}
  From \cref{eq:wpconditions_i} we have:
  \begin{equation}
    \label{eq:ciSASP}
    \begin{split}
      c_i &\le \inf_{\boldsymbol{x}_h \in \boldsymbol{X}_h} \sup_{\boldsymbol{y}_h \in \boldsymbol{X}_h}
      \frac{\langle A\boldsymbol{x}_h, \boldsymbol{y}_h\rangle_{\boldsymbol{X}_h^{\ast},\boldsymbol{X}_h}}{\norm[0]{\boldsymbol{x}_h}_{\boldsymbol{X}_h}\norm[0]{\boldsymbol{y}_h}_{\boldsymbol{X}_h}}
      \\
      &\le
      \inf_{\bar{x}_h \in \bar{X}_h}
      \sup_{(v_h,\bar{y}_h) \in X_h \times \bar{X}_h}
      \frac{\langle S_A\bar{x}_h, \bar{y}_h \rangle_{\bar{X}_h^{\ast},\bar{X}_h}}{\norm[0]{(-A_{11}^{-1}A_{21}^T\bar{x}_h,\bar{x}_h)}_{\boldsymbol{X}_h}\norm[0]{(v_h,\bar{y}_h)}_{\boldsymbol{X}_h}},
    \end{split}
  \end{equation}
  where, to obtain the second inequality, we used \cref{eq:AuhvhgeSA}
  and chose $x_h = -A_{11}^{-1}A_{21}^T\bar{x}_h$. We further observe that
  \begin{equation}
    \label{eq:Sp-minimal}
    \norm[0]{(w_h,\bar{w}_h)}_{\boldsymbol{X}_h}^2
    =\langle P\boldsymbol{w}_h, \boldsymbol{w}_h \rangle_{\boldsymbol{X}_h^{\ast},\boldsymbol{X}_h}
    \ge \langle S_P \bar{w}_h, \bar{w}_h \rangle_{\bar{X}_h^{\ast},\bar{X}_h}
    = \norm[0]{\bar{w}_h}_{\bar{X}_h}^2,
  \end{equation}
  for all $\boldsymbol{w}_h \in \boldsymbol{X}_h$,
  where the inequality follows by replacing $A$ by $P$ in
  \cref{eq:AuhvhgeSA} and using that $P_{11}$ is a symmetric positive
  operator. Combined with \cref{eq:ciSASP},
  \begin{multline}
    \label{eq:ciSASP2}
    c_i \le
    \inf_{\bar{x}_h \in \bar{X}_h}
    \sup_{\bar{y}_h \in \bar{X}_h}
    \frac{\langle S_A\bar{x}_h, \bar{y}_h \rangle_{\bar{X}_h^{\ast},\bar{X}_h}}
    {\norm[0]{\bar{x}_h}_{\bar{X}_h}\norm[0]{\bar{y}_h}_{\bar{X}_h}}
    =
    \inf_{\bar{x}_h \in \bar{X}_h}
    \sup_{\bar{y}_h \in \bar{X}_h}
    \frac{(S_P^{-1}S_A\bar{x}_h, \bar{y}_h )_{\bar{X}_h}}
    {\norm[0]{\bar{x}_h}_{\bar{X}_h}\norm[0]{\bar{y}_h}_{\bar{X}_h}}
    \\
    =
    \norm[0]{(S_P^{-1}S_A)^{-1}}_{\mathcal{L}(\bar{X}_h,\bar{X}_h)}^{-1}.
  \end{multline}
  \textbf{Step 2.} We now prove that \cref{eq:spectral-equivalence}
  implies \cref{eq:liftingcondition}. Suppose that
  \cref{eq:spectral-equivalence} and all the assumptions on $A$ hold
  but \cref{eq:liftingcondition} fails. Then, there exist sequences
  $\cbr[0]{ \boldsymbol{X}_{h_k} }_{k\ge 1}$,
  $\cbr[0]{ 0 \not = \bar{x}_{h_k} \in \bar{X}_{h_k} }_{k\ge 1}$,
  $\cbr[0]{C_k >0 }_{k\ge 1}$ such that
  $\lim_{k \to \infty} C_k = +\infty$, and
  \begin{equation*}
    \norm[0]{ (-A_{11}^{-1} A_{21}^T \bar{x}_{h_k}, \bar{x}_{h_k}) }_{\boldsymbol{X}_{h_k}} = C_k \norm[0]{ \bar{x}_{h_k} }_{\bar{X}_{h_k}}.
  \end{equation*}
  Choosing
  $\boldsymbol{x}_h = (-A_{11}^{-1}A_{21}^T \bar{x}_{h_k},
  \bar{x}_{h_k})$ in the first line of \cref{eq:ciSASP} and using
  \cref{eq:Sp-minimal},
  \begin{align*}
    c_i &\le \sup_{(v_{h_k}, \bar{y}_{h_k}) \in X_{h_k} \times \bar{X}_{h_k}}
          \frac{\langle S_A\bar{x}_{h_k}, \bar{y}_{h_k} \rangle_{\bar{X}_{h_k}^{\ast},\bar{X}_{h_k}}}
          {\norm[0]{(-A_{11}^{-1}A_{21}^T\bar{x}_{h_k}, \bar{x}_{h_k})}_{\boldsymbol{X}_{h_k}} \norm[0]{(v_{h_k}, \bar{y}_{h_k} )}_{\boldsymbol{X}_{h_k}} } 
    \\
        &\le \sup_{\bar{y}_{h_k} \in \bar{X}_{h_k}} \frac{\langle S_A\bar{x}_{h_k}, \bar{y}_{h_k} \rangle_{\bar{X}_{h_k}^{\ast},\bar{X}_{h_k}}}
          {C_k\norm[0]{\bar{x}_{h_k}}_{\bar{X}_{h_k}}\norm[0]{\bar{y}_{h_k}}_{\bar{X}_{h_k}}} 
        = \sup_{\bar{y}_{h_k} \in \bar{X}_{h_k}} \frac{( S_P^{-1} S_A\bar{x}_{h_k}, \bar{y}_{h_k} )_{\bar{X}_{h_k}}}
          {C_k\norm[0]{\bar{x}_{h_k}}_{\bar{X}_{h_k}}\norm[0]{\bar{y}_{h_k}}_{\bar{X}_{h_k}}}.
  \end{align*}
  Therefore,
  $c_i C_k \le \norm[0]{ S_P^{-1} S_A }_{\mathcal{L}(\bar{X}_{h_k},
    \bar{X}_{h_k})}$, which contradicts the assumption that
  $\norm[0]{ S_P^{-1} S_A }_{\mathcal{L}(\bar{X}_{h_k},
    \bar{X}_{h_k})}$ is uniformly bounded because
  $\lim_{k \to \infty} C_k = +\infty$.
\end{proof}

\begin{remark}
  \label{rem:SPhinnerprod}
  A similar result to \cref{thm:SpparamrobSa} also holds when $P$ is
  replaced by $\widehat{P}$, which induces an equivalent norm on
  $\boldsymbol{X}_h$ and satisfies the assumptions of $P$ in
  \cref{thm:SpparamrobSa}. The proof follows by taking the
  $\widehat{P}$-associated norm on $\boldsymbol{X}_h$.
\end{remark}

\section{Application to the Darcy and Stokes problems}
\label{s:application}
We determine a parameter-robust preconditioner for the reduced problem
of a hybridizable discretization proceeds in two steps. First, we
determine a parameter-robust preconditioner $P^{-1}$ for the full
problem by finding a norm $\tnorm{\cdot}_{\boldsymbol{X}_h}$ such that
\cref{eq:wpconditions} holds (see \cref{rem:Phparamrobust}). Second,
we use \cref{thm:SpparamrobSa} to show that $S_P^{-1}$ is a
parameter-robust preconditioner for the reduced problem by proving
\cref{eq:liftingcondition}. In this section, we demonstrate this
approach for hybridizable discretizations of the Darcy and Stokes
problems.

\subsection{Preliminaries}
\label{ss:prelim}
Let $\Omega \subset \mathbb{R}^d$ be a bounded polygonal domain with
boundary $\Gamma$. The unit outward normal to the boundary is denoted
by $n$. By $\mathcal{T}_h := \cbr[0]{K}$ we denote a family of
simplicial triangulations of the domain $\Omega$. The diameter of a
cell $K \in \mathcal{T}_h$ is denoted by $h_K$ while
$h := \max_{K \in \mathcal{T}_h}h_K$. We will assume that the mesh
consists of shape-regular cells. The boundary of a cell $K$ is denoted
by $\partial K$ and the unit outward normal vector to $\partial K$ is
denoted by $n_K$. However, we will drop the subscript $K$ where
confusion cannot occur. The set and union of faces are denoted by
$\mathcal{F}_h$ and $\Gamma^0$, respectively.

Let $D \subset \mathbb{R}^d$ be a Lipschitz domain. The Lebesgue space
of square integrable functions is denoted by $L^2(D)$ while the
$L^2$-inner product and norm on $D$ are denoted by $(\cdot, \cdot)_D$
and $\norm[0]{\cdot}_D$, respectively. The $L^2$-inner product over a
surface $S \subset \mathbb{R}^{d-1}$ is denoted by
$\langle \cdot, \cdot \rangle_S$.

We will require the following inner products:
$(\phi, \psi)_{\mathcal{T}_h} := \sum_{K\in\mathcal{T}_h}(\phi,
\psi)_K$ and
$\langle \phi, \psi \rangle_{\partial\mathcal{T}_h} :=
\sum_{K\in\mathcal{T}_h} \langle \phi, \psi \rangle_{\partial K}$ if
$\phi, \psi$ are scalar, and
$(\phi, \psi)_{\mathcal{T}_h} := \sum_{i=1}^d(\phi_i,
\psi_i)_{\Omega}$ and
$\langle \phi, \psi \rangle_{\partial \mathcal{T}_h} := \sum_{i=1}^d
\langle \phi_i, \psi_i\rangle_{\partial \mathcal{T}_h}$ if
$\phi, \psi$ are vector-valued. The norms induced by these inner
products are denoted by $\norm[0]{\cdot}_{\mathcal{T}_h}$ and
$\norm[0]{\cdot}_{\partial\mathcal{T}_h}$, respectively. These $L^2$
cell- and face-based inner products are not to be confused with the
inner product on function spaces $(\cdot, \cdot)_{\boldsymbol{X}_h}$
and duality pairing
$\langle \cdot, \cdot
\rangle_{\boldsymbol{X}_h^{\ast},\boldsymbol{X}_h}$.

We require the following
finite element function spaces:
\begin{align*}
  V_h
  &:= \cbr[1]{ v_h \in [L^2(\Omega)]^d\ :\ v_h \in [\mathbb{P}_k(K)]^d,\ \forall  K \in \mathcal{T}_h },
  \\
  \bar{V}_h
  &:= \cbr[1]{ \bar{v}_h \in [L^2(\Gamma^0)]^d\ :\ \bar{v}_h \in [\mathbb{P}_k(F)]^d,\ \forall F \in \mathcal{F}_h,\
    \bar{v}_h = 0 \text{ on } \Gamma},
  \\
  Q_h
  &:= \cbr[1]{ q_h \in L^2(\Omega)\ :\ q_h \in \mathbb{P}_{k-1}(K),\ \forall K \in \mathcal{T}_h },
  \\
  \bar{Q}_h
  &:= \cbr[1]{ \bar{q}_h \in L^2(\Gamma^0)\ :\ \bar{q}_h \in \mathbb{P}_{k}(F), \ \forall F \in \mathcal{F}_h },
  \\
  \bar{Q}_h^0
  &:= \cbr[1]{ \bar{\psi}_h \in \bar{Q}_h\ :\ \bar{\psi}_h = 0 \text{ on } \Gamma},
\end{align*}
where $\mathbb{P}_r(K)$ and $\mathbb{P}_r(F)$ denote the sets of
polynomials of degree at most $r$ on a cell $K$ and face $F$,
respectively. Cell and face function pairs will be denoted by
boldface, e.g.,
$\boldsymbol{v}_h = (v_h, \bar{v}_h) \in \boldsymbol{V}_h = V_h \times
\bar{V}_h$. Likewise, $\boldsymbol{Q}_h = Q_h \times \bar{Q}_h$ and
$\boldsymbol{Q}_h^0 = Q_h \times \bar{Q}_h^0$.

The following mesh-dependent norms are defined on $\boldsymbol{V}_h$,
$\boldsymbol{Q}_h$, $\boldsymbol{Q}_h^0$, $\bar{V}_h$, and
$\bar{Q}_h$:
\begin{subequations}
  \label{eq:norms-hdg}
  \begin{align}
    \label{eq:norm-v}
    \tnorm{\boldsymbol{v}_h}_v^2
    & := \norm[0]{\varepsilon(v_h)}^2_{\mathcal{T}_h} + \eta \norm[0]{h_K^{-1/2}(v_h - \bar{v}_h)}^2_{\partial \mathcal{T}_h}
    && \forall \boldsymbol{v}_h \in \boldsymbol{V}_h,
    \\
    \label{eq:norm-pT}
    \tnorm{\boldsymbol{q}_h}_{0, p}^2
    &:= \norm[0]{q_h}_{\mathcal{T}_h}^2 + \norm[0]{h_K^{1/2}\bar{q}_h}^2_{\partial \mathcal{T}_h}
    && \forall \boldsymbol{q}_h \in \boldsymbol{Q}_h,
    \\
    \label{eq:norm-p}
    \tnorm{\boldsymbol{q}_h}_p^2
    &:= \norm[0]{\nabla q_h}_{\mathcal{T}_h}^2 + \eta \norm[0]{h_K^{-1/2}(q_h - \bar{q}_h)}_{\partial \mathcal{T}_h}^2
    && \forall \boldsymbol{q}_h \in \boldsymbol{Q}_h^0,
    \\
    \label{eq:h-u-norm}
    \tnorm{\bar{v}_h}_{h,u}^2
    &:= \norm[0]{h_K^{-1/2}(\bar{v}_h - m_K(\bar{v}_h))}^2_{\partial \mathcal{T}_h}
    && \forall \bar{v}_h \in \bar{V}_h,
    \\
    \label{eq:h-p-norm}
    \tnorm{\bar{q}_h}_{h,p}^2
    &:= \norm[0]{h_K^{-1/2}(\bar{q}_h - m_K(\bar{q}_h))}^2_{\partial \mathcal{T}_h}
    && \forall \bar{q}_h \in \bar{Q}_h,
  \end{align}
\end{subequations}
where $\varepsilon(u) := (\nabla u + (\nabla u)^T)/2$ is the symmetric
gradient, $\eta > 0$ is a penalty parameter that will be specified
later, and where
$m_K(\bar{v}_h) := |\partial K|^{-1} \int_{\partial K} \bar{v}_h \dif
s$.

\subsection{The Darcy problem}
\label{ss:darcy}

The Darcy problem with reactive term is given by:
\begin{equation}
\label{eq:darcy-eq}
    \xi^{-1} u + \nabla p = 0
    \text{ in } \Omega,
    \quad \nabla\cdot u + \gamma p = f
    \text{ in } \Omega,   
    \quad
    p = 0 \text{ on } \Gamma,
\end{equation}
where $u$ denotes the velocity, $p$ the pressure, $f \in L^2(\Omega)$
a source term, $\xi > 0$ a constant diffusion parameter, and
$\gamma > 0$ a constant reaction parameter. In the notation of
\cref{ss:afterHyb}, let $X_h := V_h \times Q_h$,
$\bar{X}_h := \bar{Q}_h^0$, and
$\boldsymbol{X}_h := X_h \times \bar{X}_h$. The hybrid BDM
discretization of \cref{eq:darcy-eq} is given by
\cite{arnold1985mixed}: Given $f \in L^2(\Omega)$, find
$(u_h,\boldsymbol{p}_h) \in \boldsymbol{X}_h$ such that
\begin{equation}
  \label{eq:hdg2-scheme-darcy}
  a_h((u_h, \boldsymbol{p}_h), (v_h, \boldsymbol{q}_h))
  = (f, q_h)_{\mathcal{T}_h}
  \qquad \forall (v_h,\boldsymbol{q}_h) \in \boldsymbol{X}_h,
\end{equation}
where
\begin{subequations}
  \begin{align}
    \label{eq:Ddefah}
    a_h((u_h, \boldsymbol{p}_h), (v_h, \boldsymbol{q}_h))
    &:= \xi^{-1} (u_h, v_h)_{\mathcal{T}_h} + b_h(v_h, \boldsymbol{p}_{h}) - b_h(u_h, \boldsymbol{q}_h)
      + \gamma (p_h, q_h)_{\mathcal{T}_h},
    \\
    \label{eq:Ddefbh}
    b_h(v_h, \boldsymbol{q}_{h})
    &:= - (q_{h}, \nabla \cdot v_h)_{\mathcal{T}_h}
    + \langle \bar{q}_{h}, v_h \cdot n \rangle_{\partial \mathcal{T}_h}.
  \end{align}  
\end{subequations}
Consider the following inner products on $V_h$, $\boldsymbol{Q}_h^0$,
and $\boldsymbol{X}_h$:
\begin{subequations}
\begin{align}
    &(u_h,v_h)_{v,D}
    := \xi^{-1}(u_h,v_h)_{\mathcal{T}_h},
    \\
    &(\boldsymbol{p}_h,\boldsymbol{q}_h)_{q,D}
    :=\gamma (p_h,q_h)_{\mathcal{T}_h}
    + \xi (\nabla p_h,\nabla q_h)_{\mathcal{T}_h}
    + \xi \eta \langle h_K^{-1} (p_h-\bar{p}_h), q_h - \bar{q}_h\rangle_{\partial \mathcal{T}_h},
    \\
    \label{eq:darcyInnerProd}
    &( (u_h,p_h,\bar{p}_h), (v_h,q_h,\bar{q}_h) )_{\boldsymbol{X}_h}
    := (u_h,v_h)_{v,D} + (\boldsymbol{p}_h,\boldsymbol{q}_h)_{q,D},    
\end{align}
\end{subequations}
and their induced norms $\tnorm{v_h}_{v,D}$,
$\tnorm{\boldsymbol{q}_h}_{q,D}$, and
$\tnorm{(v_h,\boldsymbol{q}_h)}_{\boldsymbol{X}_h}$. In
\cref{sss:darcyStep1} we will show that the preconditioner $P^{-1}$
defined through the inner product \cref{eq:darcyInnerProd}, see
\cref{eq:defOpP}, is parameter-robust for the Darcy discretization
\cref{eq:hdg2-scheme-darcy}. In \cref{sss:darcyStep2} we then show
that the preconditioner $S_P^{-1}$, as defined by \cref{eq:defSCSP},
is parameter-robust for the reduced problem of the Darcy
discretization \cref{eq:hdg2-scheme-darcy} after eliminating $u_h$ and
$p_h$.

\subsubsection{The Darcy problem: preconditioner $P^{-1}$}
\label{sss:darcyStep1}

The following lemma shows that \cref{eq:wpconditions} is valid with
constants $c_b$ and $c_i$ independent of $h$, $\xi$, and $\gamma$.

\begin{lemma}
  \label{lem:cbciDarcy}
  For $a_h(\cdot, \cdot)$ in \cref{eq:Ddefah} and the norm
  $\tnorm{\cdot}_{\boldsymbol{X}_h}$ on $\boldsymbol{X}_h$ induced by
  the inner product defined in \cref{eq:darcyInnerProd},
  \cref{eq:wpconditions} holds with constants $c_b$ and $c_i$
  independent of $h$, $\xi$, and $\gamma$.
\end{lemma}
\begin{proof}
  See \cref{ss:Proof_cbciDarcy}.
\end{proof}

By \cref{rem:Phparamrobust} $P^{-1}$ is a parameter-robust
preconditioner for \cref{eq:hdg2-scheme-darcy}.

\subsubsection{The Darcy problem: the reduced preconditioner $S_P^{-1}$}
\label{sss:darcyStep2}

To show \cref{eq:liftingcondition} we require some preliminary
results. The following definition defines the local solver for the
Darcy problem similar to \cite[Definition
2]{rhebergen2018preconditioning}.

\begin{definition}[Local solver for the Darcy problem]
  \label{def:locSolverDarcy}
  Let $K \in \mathcal{T}_h$, and let $V(K) := [\mathbb{P}_k(K)]^d$ and
  $Q(K) := \mathbb{P}_{k-1}(K)$ be the polynomial spaces in which the
  velocity and pressure are approximated on a cell. Given
  $\bar{t}_h \in \bar{Q}_h^0$ and $s \in L^2(\Omega)$, the functions
  $u_h^L(\bar{t}_h, s)\in V_h$ and $p_h^L(\bar{t}_h, s) \in Q_h$,
  restricted to cell $K$, satisfy the local problem
  \begin{equation}
    \label{eq3:localsolver-darcy}
    a_h^K((u_h^L, p_h^L),(v_h, q_h))
    = f_h^K((v_h, q_h)) \qquad \forall (v_h, q_h) \in V(K)\times Q(K),
  \end{equation}
  where  
  \begin{equation*}
    \begin{split}
    a_h^K((u_h, p_h), (v_h, q_h))
    &:= \xi^{-1} (u_h, v_h)_{K}
    - (p_h, \nabla \cdot v_h)_K
    + (q_h, \nabla \cdot u_h)_K
    + \gamma (p_h, q_h)_K,
    \\
    f_h^K((v_h, q_h))
    &:= (s, q_h)_K
    - \langle \bar{t}_h, v_h\cdot n \rangle_{\partial K}.
    \end{split}    
  \end{equation*}
\end{definition}

The following lemma uses the local solver to eliminate $u_h$ and $p_h$
from \cref{eq:hdg2-scheme-darcy}. Since the steps of the proof are
identical to those of \cite[Lemma 4]{rhebergen2018preconditioning}, we
will omit the proof.

\begin{lemma}
  \label{lem:weak-reduced-darcy}
  Let $f \in L^2(\Omega)$ be given and let $u_h^f := u_h^L(0,f)$ and
  $p_h^f := p_h^L(0,f)$. Furthermore, define
  $l_u(\bar{q}_h) := u_h^L(\bar{q}_h,0)$ and
  $l_p(\bar{q}_h) := p_h^L(\bar{q}_h,0)$ for all
  $\bar{q}_h \in \bar{Q}_h^0$. Let $\bar{p}_h \in \bar{Q}_h^0$ be the
  solution to
  \begin{equation}
    \label{eq:weak-reduced-darcy-2}
    \bar{a}_h(\bar{p}_h, \bar{q}_h)
    = \bar{f}_h(\bar{q}_h)
    \qquad \forall \bar{q}_h \in \bar{Q}_h^0,
  \end{equation}
  where $\bar{f}_h(\bar{q}_h) := (f, l_p(\bar{q}_h))_{\mathcal{T}_h}$, and where
  \begin{equation*}
    \bar{a}_h(\bar{p}_h, \bar{q}_h) = a_h((l_u(\bar{p}_h), l_p(\bar{p}_h), \bar{p}_h), (l_u(\bar{q}_h), l_p(\bar{q}_h), \bar{q}_h).
  \end{equation*}
  Then $(u_h,p_h,\bar{p}_h)$, where $u_h = u_h^f + l_u(\bar{p}_h)$ and
  $p_h = p_h^f + l_p(\bar{p}_h)$, solves \cref{eq:hdg2-scheme-darcy}.
\end{lemma}

\begin{remark}
  \Cref{eq:weak-reduced-darcy-2} is the reduced form of
  \cref{eq:hdg2-scheme-darcy} in a variational setting. In operator
  form, it corresponds to \cref{eq:SCop}.
\end{remark}

Next, we require some results related to the HDG discretization of the
following auxiliary problem:
\begin{equation*}
  -\nabla \cdot (\xi \nabla p) + \gamma p = f \text{ in } \Omega,
  \quad
  p = 0 \text{ on } \Gamma.  
\end{equation*}
The HDG discretization of this problem is given by
\cite{wells2011analysis}: Given $f \in L^2(\Omega)$, find
$\boldsymbol{p}_h \in \boldsymbol{Q}_h^0$ such that
\begin{equation}
  \label{eq:HDGaux}
  \tilde{a}_h(\boldsymbol{p}_h, \boldsymbol{q}_h) = (f, q_h)_{\mathcal{T}_h} \quad \forall \boldsymbol{q}_h \in \boldsymbol{Q}_h^0,
\end{equation}
where
\begin{multline}
  \label{eq:HDGaux_ah}
  \tilde{a}_h(\boldsymbol{p}_h, \boldsymbol{q}_h)
  := \xi (\nabla p_h, \nabla q_h)_{\mathcal{T}_h}
  + \xi \eta \langle h_K^{-1}(p_h-\bar{p}_h), q_h-\bar{q}_h \rangle_{\partial\mathcal{T}_h}
  + \gamma(p_h,q_h)_{\mathcal{T}_h}
  \\
  - \xi \langle \nabla p_n \cdot n, q_h - \bar{q}_h \rangle_{\partial\mathcal{T}_h}
  - \xi \langle \nabla q_n \cdot n, p_h - \bar{p}_h \rangle_{\partial\mathcal{T}_h}.
\end{multline}
Straightforward modifications to the proofs of \cite[Lemmas 5.2 and
5.3]{wells2011analysis} lead to:
\begin{equation}
  \label{eq:coerbndatildeD}
  C_1 \tnorm{\boldsymbol{q}_h}_{q,D}^2 \le \tilde{a}_h(\boldsymbol{q}_h,\boldsymbol{q}_h) \le C_2 \tnorm{\boldsymbol{q}_h}_{q,D}^2
  \quad \forall \boldsymbol{q}_h \in \boldsymbol{Q}_h^0.
\end{equation}
To eliminate $p_h$ from \cref{eq:HDGaux} we introduce the following
local solver.

\begin{definition}[Local solver for the auxiliary problem]
  \label{def:localsolverAuxProb}
  Let $K \in \mathcal{T}_h$, and let $Q(K) := \mathbb{P}_{k-1}(K)$ be
  the polynomial space in which the pressure is approximated on a
  cell. Given $\bar{t}_h \in \bar{Q}_h^0$ and $s \in L^2(\Omega)$, the
  function $\tilde{p}_h^L(\bar{t}_h,s) \in Q_h$, restricted to cell
  $K$, satisfies the local problem
  \begin{equation}
    \label{eq:localProbAuxD}
    \tilde{a}_h^K(\tilde{p}_h^L, q_h) = \tilde{f}_h^K(q_h) \quad \forall q_h \in Q(K),
  \end{equation}
  where
  \begin{align*}
    \tilde{a}_h^K(p_h,q_h)
    :=& \xi (\nabla p_h, \nabla q_h)_K + \xi \eta h_K^{-1} \langle p_h, q_h \rangle_{\partial K} + \gamma (p_h,q_h)_K
    \\
      &- \xi \langle \nabla p_h\cdot n, q_h \rangle_{\partial K} - \xi \langle \nabla q_h \cdot n, p_h \rangle_{\partial K},
    \\
    \tilde{f}_h^K(q_h)
    :=& (s,q_h)_K
        + \xi \eta h_K^{-1} \langle q_h, \bar{t}_h \rangle_{\partial K}
        - \xi \langle \nabla q_h \cdot n, \bar{t}_h \rangle_{\partial K}.
  \end{align*}
\end{definition}

The following lemma now eliminates $p_h$ from \cref{eq:HDGaux}. We
again omit the proof because the steps are identical to those of
\cite[Lemma 4]{rhebergen2018preconditioning}.

\begin{lemma}
  \label{lem:weak-reduced-darcy-aux}
  Let $f \in L^2(\Omega)$ be given, and let 
  $\tilde{p}_h^f := \tilde{p}_h^L(0,f)$. Furthermore, define
  $\tilde{l}_p(\bar{q}_h) := \tilde{p}_h^L(\bar{q}_h,0)$ for all
  $\bar{q}_h \in \bar{Q}_h^0$. Let $\bar{p}_h \in \bar{Q}_h^0$ be the
  solution to
  \begin{equation}
    \label{eq:weak-reduced-darcy-2-aux}
    \tilde{a}_h( (\tilde{l}_p(\bar{p}_h), \bar{p}_h), (\tilde{l}_p(\bar{q}_h), \bar{q}_h) )
    = (f, \tilde{l}_p(\bar{q}_h))
    \qquad \forall \bar{q}_h \in \bar{Q}_h^0.
  \end{equation}
  Then $(p_h,\bar{p}_h)$, where
  $p_h = \tilde{p}_h^f + \tilde{l}_p(\bar{p}_h)$, solves
  \cref{eq:HDGaux}.
\end{lemma}

We now have the following lemma that relates the reduced form of
\cref{eq:HDGaux_ah} to the norm $\tnorm{\cdot}_{\boldsymbol{X}_h}$.

\begin{lemma}
  \label{lem:mainLemmaStep2Darcy}
  For any $\bar{q}_h \in \bar{Q}_h^0$ there exists a positive uniform
  constant $c_d$ such that
  \begin{equation}
    \label{eq:mainEqStep2Darcy}
    \tnorm{(l_u(\bar{q}_h), l_p(\bar{q}_h), \bar{q}_h)}_{\boldsymbol{X}_h}^2
    \le
    c_d \tilde{a}_h ( (\tilde{l}_p(\bar{q}_h), \bar{q}_h), (\tilde{l}_p(\bar{q}_h), \bar{q}_h) ).
  \end{equation}
\end{lemma}
\begin{proof}
  See \cref{ss:Proof_mainLemmaStep2Darcy}.
\end{proof}

We can now prove \cref{eq:liftingcondition} for the Darcy problem.

\begin{theorem}
  \label{thm:uniformconstantDarcy}
  There exists a uniform constant $c_l > 0$ such that
  \cref{eq:liftingcondition} holds for the reduced form of
  \cref{eq:hdg2-scheme-darcy}.
\end{theorem}
\begin{proof}
  \textbf{Step (i).} Let
  $\tilde{A}:\boldsymbol{Q}_h^0 \to \boldsymbol{Q}_h^{0,\ast}$ be the
  operator such that
  $\langle \tilde{A} \boldsymbol{p}_h, \boldsymbol{q}_h
  \rangle_{\boldsymbol{Q}_h^{\ast},\boldsymbol{Q}_h^0} =
  \tilde{a}_h(\boldsymbol{p}_h, \boldsymbol{q}_h)$ for all
  $\boldsymbol{p}_h,\boldsymbol{q}_h \in \boldsymbol{Q}_h^0$. We can
  write \cref{eq:HDGaux} as
  $\tilde{A}\boldsymbol{p}_h = \boldsymbol{f}_h$ in
  $\boldsymbol{Q}_h^0$. Since
  $\boldsymbol{Q}_h^0 := Q_h \times \bar{Q}_h^0$, we may further write
  $\tilde{A}\boldsymbol{p}_h = \boldsymbol{f}_h$ in block matrix form
  \begin{equation*}
    \begin{bmatrix}
      \tilde{A}_{11} & \tilde{A}_{21}^T
      \\
      \tilde{A}_{21} & \tilde{A}_{22}
    \end{bmatrix}
    \begin{bmatrix}
      p_h \\
      \bar{p}_h
    \end{bmatrix}
    =
    \begin{bmatrix}
      f_h
      \\
      0
    \end{bmatrix}.
  \end{equation*}
  From the bilinear form we remark that $\tilde{A}_{11}$ is
  invertible. Then, eliminating $p_h$, we find the following equation
  for $\bar{p}_h$:
  \begin{equation}
    \label{eq:darcySAaux}
    S_{\tilde{A}} \bar{p}_h = \bar{b}_h
    \quad \text{ where } \quad
    S_{\tilde{A}} = \tilde{A}_{22} - \tilde{A}_{21}\tilde{A}_{11}^{-1}\tilde{A}_{21}^T,
    \quad
    \bar{b}_h = - \tilde{A}_{21}\tilde{A}_{11}^{-1}f_h.
  \end{equation}
  Replacing $A$ in \cref{eq:AuhvhgeSA} by $\tilde{A}$ and using that
  $\tilde{A}_{11}$ is symmetric and positive, we find:
  \begin{equation}
    \label{eq:tildeAresultD}
    \langle \tilde{A}\boldsymbol{q}_h, \boldsymbol{q}_h \rangle_{\boldsymbol{Q}_h^{0,\ast},\boldsymbol{Q}_h^0}
    \ge \langle S_{\tilde{A}}\bar{q}_h, \bar{q}_h \rangle_{\bar{Q}_h^{0,\ast}, \bar{Q}_h^0}
    \quad \forall \boldsymbol{q}_h=(q_h,\bar{q}_h) \in \boldsymbol{Q}_h^0.
  \end{equation}
  \textbf{Step (ii).} Consider the inner product
  $(\cdot, \cdot)_{\boldsymbol{X}_h}$ defined in
  \cref{eq:darcyInnerProd}, and define the preconditioner
  $P : \boldsymbol{X}_h \to \boldsymbol{X}_h^{\ast}$, and the
  operators $P^u : V_h \to V_h^{\ast}$ and
  $P^p:\boldsymbol{Q}_h^0 \to \boldsymbol{Q}_h^{0,\ast}$ such that
  \begin{equation*}
    \begin{split}
      \langle P (u_h,\boldsymbol{p}_h), (v_h,\boldsymbol{q}_h)
      \rangle_{\boldsymbol{X}_h^{\ast},\boldsymbol{X}_h}
      &=( (u_h,p_h,\bar{p}_h), (v_h,q_h,\bar{q}_h) )_{\boldsymbol{X}_h}
      =(u_h,v_h)_{v,D} + (\boldsymbol{p}_h,\boldsymbol{q}_h)_{q,D}
      \\
      &=\langle P^u u_h, v_h \rangle_{V_h^{\ast},V_h}
      + \langle P^p \boldsymbol{p}_h, \boldsymbol{q}_h \rangle_{\boldsymbol{Q}_h^{0,\ast},\boldsymbol{Q}_h}.
    \end{split}
  \end{equation*}
  We note that $P$ has the following structure:
  \begin{equation*}
    P =
    \begin{bmatrix}
      P^u & 0
      \\
      0 & P^p
    \end{bmatrix}
    =
    \begin{bmatrix}
      P^u & 0 & 0
      \\
      0 & P_{11}^p & (P_{21}^p)^T
      \\
      0 & P_{21}^p & P_{22}^p
    \end{bmatrix},
  \end{equation*}
  with $P_{11}^p : Q_h \to Q_h^{\ast}$,
  $P_{21}^p : Q_h \to \bar{Q}_h^{0,\ast}$, and
  $P_{22}^p : \bar{Q}_h^0 \to \bar{Q}_h^{0,\ast}$. We associate with
  $P$ the operator $S_P : \bar{Q}_h^0 \to \bar{Q}_h^{0,\ast}$ defined
  by $S_P := P_{22}^p - P_{21}^p(P_{11}^p)^{-1}(P_{21}^p)^T$.
  \\
  \textbf{Step (iii).} By \cref{eq:coerbndatildeD} and
  \cref{eq:tildeAresultD} we find that
  \begin{equation}
    \label{eq:QqhqhSAqbqbn}
    C_2 \langle P^p \boldsymbol{q}_h, \boldsymbol{q}_h \rangle_{\boldsymbol{Q}_h^{0,\ast},\boldsymbol{Q}_h^0}
    \ge \langle S_{\tilde{A}}\bar{q}_h, \bar{q}_h \rangle_{\bar{Q}_h^{0,\ast}, \bar{Q}_h^0}
    \quad \forall \boldsymbol{q}_h \in \boldsymbol{Q}_h^0.
  \end{equation}
  Similar to \cref{eq:AuhvhgeSA}, we have
  \begin{multline}
    \label{eq:AuhvhgeSAQn}
    \langle P^p\boldsymbol{q}_h, \boldsymbol{q}_h\rangle_{\boldsymbol{Q}_h^{0,\ast},\boldsymbol{Q}_h^0}
    \\
    =
    \langle P_{11}^p(q_h + (P_{11}^p)^{-1}(P_{21}^p)^T\bar{q}_h), q_h + (P_{11}^p)^{-1}(P_{21}^p)^{T}\bar{q}_h \rangle_{Q_h^{0,\ast},Q_h^0}
    + \langle S_P\bar{q}_h, \bar{q}_h \rangle_{\bar{Q}_h^{0,\ast},\bar{Q}_h^0}.
  \end{multline}
  Therefore, choosing $q_h = -P_{11}^{-1}P_{21}^T\bar{q}_h$ in
  \cref{eq:QqhqhSAqbqbn}, we obtain:
  \begin{equation}
    \label{eq:QqhqhSAqbqbQchosenn}
    C_2 \langle S_P\bar{q}_h, \bar{q}_h \rangle_{\bar{Q}_h^{0,\ast},\bar{Q}_h^0}
    \ge \langle S_{\tilde{A}}\bar{q}_h, \bar{q}_h \rangle_{\bar{Q}_h^{0,\ast}, \bar{Q}_h^0}
    \quad \forall \bar{q}_h \in \bar{Q}_h^0.
  \end{equation}
  \textbf{Step (iv).} Let
  $A : \boldsymbol{X}_h \to \boldsymbol{X}_h^{\ast}$ be the operator
  such that
  $\langle A \boldsymbol{x}_h, \boldsymbol{y}_h
  \rangle_{\boldsymbol{X}_h^{\ast},\boldsymbol{X}_h} =
  a_h(\boldsymbol{x}_h, \boldsymbol{y}_h)$ for all
  $\boldsymbol{x}_h, \boldsymbol{y}_h \in \boldsymbol{X}_h$. We can
  then write \cref{eq:hdg2-scheme-darcy} as
  $A \boldsymbol{x}_h = \boldsymbol{f}_h$ in $\boldsymbol{X}_h$. Here
  we write $\boldsymbol{X}_h := (V_h \times Q_h) \times \bar{Q}_h^0$
  and so rewrite $A\boldsymbol{x}_h = \boldsymbol{f}_h$ as
  \begin{equation*}
    \begin{bmatrix}
      A_{11} & A_{21}^T
      \\
      A_{21} & A_{22}
    \end{bmatrix}
    \begin{bmatrix}
      z_h \\ \bar{p}_h
    \end{bmatrix}
    =
    \begin{bmatrix}
      f_h \\ 0
    \end{bmatrix},
  \end{equation*}
  where $z_h = (u_h, p_h)$. Eliminating $z_h$, we find the following
  equation for $\bar{p}_h$:
  \begin{equation}
    \label{eq:darcySA}
    S_{A} \bar{p}_h = \bar{b}_h
    \quad \text{ where } \quad
    S_{A} = A_{22} - A_{21}A_{11}^{-1}A_{21}^T,
    \quad
    \bar{b}_h = - A_{21}A_{11}^{-1}f_h.
  \end{equation}
  \Cref{lem:mainLemmaStep2Darcy} in operator form states:
  \begin{equation*}
    \norm[0]{(-A_{11}^{-1}A_{21}^T\bar{q}_h,\bar{q}_h)}_{\boldsymbol{X}_h}^2
    \le c_d \langle S_{\tilde{A}}\bar{q}_h, \bar{q}_h \rangle_{\bar{Q}_h^{0,\ast},\bar{Q}_h^0}
    \quad \forall \bar{q}_h \in \bar{Q}_h^0.
  \end{equation*}
  Combining this with \cref{eq:QqhqhSAqbqbQchosenn}, we obtain:
  \begin{equation*}
    \norm[0]{(-A_{11}^{-1}A_{21}^T\bar{q}_h,\bar{q}_h)}_{\boldsymbol{X}_h}^2
    \le c_dC_2 \langle S_P\bar{q}_h, \bar{q}_h \rangle_{\bar{Q}_h^{0,\ast},\bar{Q}_h^0}
    \quad \forall \bar{q}_h \in \bar{Q}_h^0.
  \end{equation*}
  Since
  $\langle S_P\bar{q}_h, \bar{q}_h
  \rangle_{\bar{Q}_h^{0,\ast},\bar{Q}_h^0} =
  (\bar{q}_h,\bar{q}_h)_{\bar{Q}_h^0} =
  \norm[0]{\bar{q}_h}_{\bar{Q}_h^0}^2$, the result follows.
\end{proof}

That $S_P^{-1}$ is a parameter-robust preconditioner for the reduced
problem of the Darcy discretization \cref{eq:hdg2-scheme-darcy} is now
a consequence of \cref{thm:SpparamrobSa}.

\subsubsection{A counter example}
\label{sss:counter}

In this section we briefly discuss an example of a preconditioner
$(P')^{-1}$ that is parameter-robust for \cref{eq:hdg2-scheme-darcy}
but for which the reduced preconditioner $S_{P'}^{-1}$ is not
parameter-robust for the reduced form of \cref{eq:hdg2-scheme-darcy}.

Consider the following inner products on $V_h$ and
$\boldsymbol{Q}_h^0$:
\begin{equation*}
  \begin{split}
    (u_h,v_h)_{v,D'}
    :=& \xi^{-1}(u_h,v_h)_{\mathcal{T}_h} + M^{-1}(\nabla \cdot u_h, \nabla \cdot v_h)_{\mathcal{T}_h}
    + \xi^{-1}\langle h_F^{-1} \jump{u_h \cdot n}, \jump{v_h \cdot n} \rangle_{\mathcal{F}_h\backslash \Gamma},
    \\
    (\boldsymbol{p}_h,\boldsymbol{q}_h)_{q,D'}
    :=&M (p_h,q_h)_{\mathcal{T}_h}
    + \xi \langle h_K \bar{p}_h, \bar{q}_h\rangle_{\partial \mathcal{T}_h},
  \end{split}
\end{equation*}
where $\jump{\cdot}$ is the usual DG jump operator and
$M := \max(\xi, \gamma)$, and define the following inner product on
$\boldsymbol{X}_h$:
\begin{equation}
  \label{eq:darcyInnerProdalt}
  ( (u_h,p_h,\bar{p}_h), (v_h,q_h,\bar{q}_h) )_{\boldsymbol{X}_h'}
  := (u_h,v_h)_{v,D'} + (\boldsymbol{p}_h,\boldsymbol{q}_h)_{q,D'}.
\end{equation}
The norm induced by this inner product is denoted by
$\tnorm{\cdot}_{\boldsymbol{X}_h'}$. It can be shown that the
boundedness and coercivity constants in \cref{eq:wpconditions} are
uniform when using this norm. As such, the preconditioner $(P')^{-1}$
associated with the inner product defined in
\cref{eq:darcyInnerProdalt} is a parameter-robust preconditioner for
\cref{eq:hdg2-scheme-darcy}.

Eliminating the cell velocity and pressure degrees-of-freedom, the
reduced preconditioner $S_{P'}^{-1}$ is defined by
\begin{equation*}
  \langle S_{P'} \bar{q}_h, \bar{q}_h \rangle_{\bar{Q}_h^{0,\ast},\bar{Q}_h^0}
  = \norm[0]{\bar{q}_h}_{\bar{Q}_h^{0'}}^2
  \qquad
  \text{where}\qquad
  \norm[0]{\bar{q}_h}_{\bar{Q}_h^{0'}}^2 := \xi \norm[0]{h_K^{1/2}\bar{q}_h}_{\partial \mathcal{T}_h}^2.
\end{equation*}
By \cref{thm:uniformconstantDarcy} and
\cite[eq. (2.9)]{gopalakrishnan2003schwarz} we find:
\begin{equation*}
  \norm[0]{(-A_{11}^{-1}A_{21}^T\bar{q}_h,\bar{q}_h)}_{\boldsymbol{X}_h}^2
  \le C \norm[0]{\bar{q}_h}_{\bar{Q}_h^0}^2
  \le C h^{-2} \norm[0]{\bar{q}_h}_{\bar{Q}_h^{0'}}^2
  \quad \forall
  \bar{q}_h \in \bar{Q}_h^0. 
\end{equation*}
It is unclear whether this bound can be improved. This bound indicates
that the reduced preconditioner $S_{P'}^{-1}$ may not be robust with
mesh refinement for the reduced form on
\cref{eq:hdg2-scheme-darcy}. We indeed demonstrate this by numerical
experiment in \cref{sss:notsatisfied}.

\subsection{The Stokes problem}
\label{ss:stokes}

The Stokes problem is given by:
\begin{equation}
\label{eq:stokes-eq}
    -\nabla \cdot 2\nu\varepsilon(u) + \nabla p = f \text{ in } \Omega,
    \quad
    \nabla \cdot u = 0 \text{ in } \Omega,
    \quad
    u = 0 \text{ on } \Gamma,    
\end{equation}
and we impose $\int_{\Omega}p \dif x = 0$. In these equations, $u$
denotes the velocity, $p$ the (kinematic) pressure,
$f \in [L^2(\Omega)]^d$ is a source term, and $\nu > 0$ is the
kinematic viscosity. Let
$\boldsymbol{X}_h := \boldsymbol{V}_h \times \boldsymbol{Q}_h$. We
consider the hybridizable discontinuous Galerkin (HDG) discretization
of \cref{eq:stokes-eq} given in \cite{rhebergen2017analysis}: Find
$(\boldsymbol{u}_h, \boldsymbol{p}_h) \in \boldsymbol{X}_h$ such that
\begin{equation}
  \label{eq:hdg-scheme-stokes}
  a_h((\boldsymbol{u}_h,\boldsymbol{p}_h), (\boldsymbol{v}_h,\boldsymbol{q}_h))
  = (f, v_h)_{\mathcal{T}_h}
  \qquad \forall (\boldsymbol{v}_h, \boldsymbol{q}_h) \in \boldsymbol{X}_h,
\end{equation}
where
\begin{subequations}
  \begin{align}
    \label{eq:defah-stokes}
    a_h((\boldsymbol{u}_h,\boldsymbol{p}_h), (\boldsymbol{v}_h,\boldsymbol{q}_h))
    :=& c_h(\boldsymbol{u}_h, \boldsymbol{v}_h) + b_h(v_h, \boldsymbol{p}_{h}) + b_h(u_h, \boldsymbol{q}_{h}),
    \\
    \label{eq:defch-stokes}
    c_h(\boldsymbol{u}_h, \boldsymbol{v}_h)
    :=& \nu(\varepsilon(u_h), \varepsilon(v_h))_{\mathcal{T}_h}
      + \nu \langle \eta h_K^{-1} (u_h - \bar{u}_h), v_h - \bar{v}_h\rangle_{\partial\mathcal{T}_h}      
    \\ \nonumber
    & - \nu \langle u_h - \bar{u}_h, \varepsilon( v_h ) n\rangle_{\partial\mathcal{T}_h}
      - \nu \langle \varepsilon( u_h) n, v_h - \bar{v}_h \rangle_{\partial\mathcal{T}_h}     
  \end{align}
\end{subequations}
and with $b_h(\cdot, \cdot)$ defined in \cref{eq:Ddefbh}.

Consider the following inner products on $\boldsymbol{V}_h$,
$\boldsymbol{Q}_h$, and $\boldsymbol{X}_h$:
\begin{subequations}
\begin{align}
  &(\boldsymbol{u}_h, \boldsymbol{v}_h)_{v,S}
  :=\nu(\varepsilon(u_h), \varepsilon(v_h))_{\mathcal{T}_h}
    + \nu \eta \langle h_K^{-1} (u_h-\bar{u}_h), v_h-\bar{v}_h\rangle_{\partial \mathcal{T}_h},
  \\
  &(\boldsymbol{p}_h, \boldsymbol{q}_h)_{q,S}
    :=\nu^{-1}(p_h, q_h)_{\mathcal{T}_h}
    + \nu^{-1}\eta^{-1}\langle h_K \bar{p}_h, \bar{q}_h\rangle_{\partial \mathcal{T}_h},
    \\
  \label{eq:innerproductXhStokes}
  &((\boldsymbol{u}_h,\boldsymbol{p}_h),
  (\boldsymbol{v}_h,\boldsymbol{q}_h))_{\boldsymbol{X}_h} :=
  (\boldsymbol{u}_h, \boldsymbol{v}_h)_{v,S} + (\boldsymbol{p}_h,
  \boldsymbol{q}_h)_{q,S},    
\end{align}
\end{subequations}
and their induced norms $\tnorm{\boldsymbol{v}_h}_{v,S}$,
$\tnorm{\boldsymbol{q}_h}_{q,S}$, and
\begin{equation}
  \label{eq:normforStokes}
  \tnorm{(\boldsymbol{v}_h, \boldsymbol{q}_h)}_{\boldsymbol{X}_h}
  := \tnorm{\boldsymbol{v}_h}_{v,S} + \tnorm{\boldsymbol{q}_h}_{q,S}.
\end{equation}
In
\cref{sss:stokesStep1} we will show that the preconditioner $P^{-1}$
defined through the inner product \cref{eq:innerproductXhStokes}, see
\cref{eq:defOpP}, is parameter-robust for the Stokes discretization
\cref{eq:hdg-scheme-stokes}. In \cref{sss:stokesStep2} we then show
that the preconditioner $S_P^{-1}$, as defined by \cref{eq:defSCSP},
is parameter-robust for the reduced problem of the Stokes
discretization \cref{eq:hdg-scheme-stokes}.

\subsubsection{The Stokes problem: preconditioner $P^{-1}$}
\label{sss:stokesStep1}

We first recall the following result from \cite[Lemma
4.3]{rhebergen2017analysis}, (assuming that $\eta > 1$): there exists
a positive uniform constant $c_1$ such that
\begin{equation}
  \label{eq:chboundStokes}
  c_h(\boldsymbol{u}_h, \boldsymbol{v}_h) \le c_1 \tnorm{\boldsymbol{u}_h}_{v,S} \tnorm{\boldsymbol{v}_h}_{v,S}.
\end{equation}
We now have the following result.

\begin{lemma}
  \label{lem:cbciStokes}
  Let $\boldsymbol{X}_h := \boldsymbol{V}_h \times \boldsymbol{Q}_h$
  and let $a_h(\cdot, \cdot)$ be as defined in
  \cref{eq:defah-stokes}. If we endow $\boldsymbol{X}_h$ with the norm
  defined in \cref{eq:normforStokes}, then \cref{eq:wpconditions}
  holds with constants $c_b$ and $c_i$ independent of $h$ and $\nu$.
\end{lemma}
\begin{proof}
  See \cref{ss:Proof_cbciStokes}.
\end{proof}

By \cref{rem:Phparamrobust} we therefore find that $P^{-1}$ is a
parameter-robust preconditioner for \cref{eq:hdg-scheme-stokes}.

\subsubsection{The Stokes problem: the reduced preconditioner $S_P^{-1}$}
\label{sss:stokesStep2}

To show \cref{eq:liftingcondition} we require some preliminary
results. We start by defining the local solver for the Stokes problem
similar to \cite[Definition 2]{rhebergen2018preconditioning}.

\begin{definition}[Local solver for the Stokes problem]
  \label{def:locSolverStokes}
  Let $K \in \mathcal{T}_h$, and let $V(K) := [\mathbb{P}_k(K)]^d$ and
  $Q(K) := \mathbb{P}_{k-1}(K)$ be the polynomial spaces in which the
  velocity and pressure are approximated on a cell. Given
  $(\bar{m}_h,\bar{t}_h) \in \bar{V}_h \times \bar{Q}_h$ and
  $s \in [L^2(\Omega)]^d$, the functions
  $u_h^L(\bar{m}_h, \bar{t}_h, s)\in V_h$ and
  $p_h^L(\bar{m}_h, \bar{t}_h, s) \in Q_h$, restricted to cell $K$,
  satisfy the local problem
  \begin{equation}
    \label{eq3:localsolver-stokes}
    a_h^K((u_h^L, p_h^L), (v_h, q_h))
    = f_h^K(v_h) \qquad \forall (v_h, q_h) \in V(K) \times Q(K),
  \end{equation}
  where  
  \begin{equation*}
    \begin{split}
      &a_h^K((u_h, p_h), (v_h, q_h))
      := \nu (\varepsilon(u_h), \varepsilon(v_h))_{K}
      + \nu \eta h_K^{-1} \langle u_h, v_h \rangle_{\partial K}          
      - \nu \langle \varepsilon(u_h)n, v_h \rangle_{\partial K}
      \\
      & \hspace{10em}
      - \nu \langle \varepsilon(v_h)n, u_h \rangle_{\partial K}
      - (p_h, \nabla \cdot v_h)_K
      - (q_h, \nabla \cdot u_h)_K, 
      \\
    &f_h^K(v_h)
    := (s, v_h)_K
    - \nu \langle \varepsilon(v_h)n, \bar{m}_h \rangle_{\partial K}
    + \nu \eta h_K^{-1} \langle \bar{m}_h, v_h \rangle_{\partial K}
    - \langle \bar{t}_h, v_h \cdot n \rangle_{\partial K}.      
    \end{split}
  \end{equation*}
\end{definition}

Using the identical steps as in the proof of \cite[Lemma
4]{rhebergen2018preconditioning}, the following lemma shows that we
can eliminate $u_h$ and $p_h$ from \cref{eq:hdg-scheme-stokes} using
the local solvers defined in \cref{def:locSolverStokes}.

\begin{lemma}
  \label{lem:weak-reduced-stokes}
  Let $f \in L^2(\Omega)$ be given and let $u_h^f := u_h^L(0,0,f)$ and
  $p_h^f := p_h^L(0,0,f)$. Furthermore, define
  $l_u(\bar{v}_h, \bar{q}_h) := u_h^L(\bar{v}_h, \bar{q}_h, 0)$ and
  $l_p(\bar{v}_h, \bar{q}_h) := p_h^L(\bar{v}_h, \bar{q}_h, 0)$ for
  all $(\bar{v}_h, \bar{q}_h) \in \bar{V}_h \times \bar{Q}_h$. Let
  $(\bar{u}_h, \bar{p}_h) \in \bar{V}_h \times \bar{Q}_h$ be the
  solution to
  \begin{equation}
    \label{eq:weak-reduced-stokes-2}
    \bar{a}_h((\bar{u}_h, \bar{p}_h), (\bar{v}_h, \bar{q}_h))
    = \bar{f}_h(\bar{v}_h, \bar{q}_h)
    \qquad \forall (\bar{v}_h, \bar{q}_h) \in \bar{V}_h \times \bar{Q}_h,
  \end{equation}
  where
  $\bar{f}_h(\bar{v}_h, \bar{q}_h) := (f, l_u(\bar{v}_h,
  \bar{q}_h))_{\mathcal{T}_h}$, and where
  \begin{multline*}
    \bar{a}_h((\bar{u}_h, \bar{p}_h), (\bar{v}_h, \bar{q}_h))
    =
    c_h((l_u(\bar{u}_h, \bar{p}_h), \bar{u}_h), (l_u(\bar{v}_h, \bar{q}_h)))
    \\
    + b_h(l_u(\bar{v}_h, \bar{q}_h), (l_p(\bar{u}_h, \bar{p}_h), \bar{p}_h))
    + b_h(l_u(\bar{u}_h, \bar{p}_h), (l_p(\bar{v}_h, \bar{q}_h), \bar{q}_h)).
  \end{multline*}
  Then $(u_h,\bar{u}_h, p_h,\bar{p}_h)$, where
  $u_h = u_h^f + l_u(\bar{u}_h, \bar{p}_h)$ and
  $p_h = p_h^f + l_p(\bar{u}_h, \bar{p}_h)$, solves
  \cref{eq:hdg-scheme-stokes}.
\end{lemma}

The following result relates the reduced form of
\cref{eq:hdg-scheme-stokes} to the norm
$\tnorm{\cdot}_{\boldsymbol{X}_h}$.

\begin{lemma}
  \label{lem:mainLemmaStep2Stokes}
  Consider any $(\bar{v}_h, \bar{q}_h) \in \bar{V}_h \times \bar{Q}_h$
  and let $l_u(\bar{v}_h,\bar{q}_h)$ and $l_p(\bar{v}_h,\bar{q}_h)$ be
  as defined in \cref{lem:weak-reduced-stokes}. Furthermore, let
  $l_u(\bar{v}_h) := u_h^L(\bar{v}_h,0,0)$. There exists a positive
  uniform constant $c_s$ such that
  \begin{multline*}
    \tnorm{(l_u(\bar{v}_h,\bar{q}_h), \bar{v}_h, l_p(\bar{v}_h,\bar{q}_h), \bar{q}_h)}_{\boldsymbol{X}_h}^2
    \\
    \le c_s \eta \del[2]{ \eta c_h((l_u(\bar{v}_h), \bar{v}_h), (l_u(\bar{v}_h), \bar{v}_h))
    + \nu^{-1}\eta^{-1}\norm[0]{h_K^{1/2}\bar{q}_h}_{\partial \mathcal{T}_h}^2 }.
  \end{multline*}
\end{lemma}
\begin{proof}
  See \cref{ss:Proof_mainLemmaStep2Stokes}.
\end{proof}

We can now prove \cref{eq:liftingcondition} for the Stokes problem.

\begin{theorem}
  There exists a constant $c_l > 0$, independent of $h$ and the model
  parameters, but dependent on the penalty parameter $\eta$, such that
  \cref{eq:liftingcondition} holds for the reduced form of
  \cref{eq:hdg-scheme-stokes}.
\end{theorem}
\begin{proof}
  \textbf{Step (i).} Let
  $C:\boldsymbol{V}_h \to \boldsymbol{V}_h^{\ast}$ be the operator
  such that
  $\langle C \boldsymbol{u}_h, \boldsymbol{v}_h
  \rangle_{\boldsymbol{V}_h^{\ast},\boldsymbol{V}_h} =
  c_h(\boldsymbol{u}_h, \boldsymbol{v}_h)$ for all
  $\boldsymbol{u}_h,\boldsymbol{v}_h \in \boldsymbol{V}_h$. Since
  $\boldsymbol{V}_h := V_h \times \bar{V}_h$, we may write
  \begin{equation*}
    C = 
    \begin{bmatrix}
      C_{11} & C_{21}^T
      \\
      C_{21} & C_{22}
    \end{bmatrix},
  \end{equation*}
  where $C_{11} : V_h \to V_h^{\ast}$,
  $C_{21} : V_h \to \bar{V}_h^{\ast}$, and
  $C_{22} : \bar{V}_h \to \bar{V}_h^{\ast}$. Associated with $C$ we
  define the operator $S_C : \bar{V}_h \to \bar{V}_h^{\ast}$ by
  \begin{equation}
    \label{eq:stokesdifS}
    S_C = C_{22} - C_{21}C_{11}^{-1}C_{21}^T.
  \end{equation}
  Replacing $A$ in \cref{eq:AuhvhgeSA} by $C$ and using that
  $\tilde{C}_{11}$ is symmetric and positive, we find:
  \begin{equation}
    \label{eq:CresultS}
    \langle C\boldsymbol{u}_h, \boldsymbol{v}_h \rangle_{\boldsymbol{V}_h^{\ast},\boldsymbol{V}_h}
    \ge \langle S_C \bar{u}_h, \bar{u}_h \rangle_{\bar{V}_h^{\ast}, \bar{V}_h}
    \quad \forall \boldsymbol{u}_h=(u_h,\bar{u}_h),\ \boldsymbol{v}_h=(v_h,\bar{v}_h) \in \boldsymbol{V}_h.
  \end{equation}
  \textbf{Step (ii).} Let $B_{11}:V_h \to Q_h^*$ and
  $B_{21}:V_h \to \bar{Q}_h^*$ be the operators such that
  $\langle (B_{11} u_h, B_{21} u_h), \boldsymbol{q}_h
  \rangle_{\boldsymbol{Q}_h^*, \boldsymbol{Q}_h}= b_h(u_h,
  \boldsymbol{q}_h)$ for all $u_h\in V_h$ and
  $\boldsymbol{q}_h \in \boldsymbol{Q}_h$. Next, Let
  $A : \boldsymbol{X}_h \to \boldsymbol{X}_h^{\ast}$ be the operator
  such that
  $\langle A(\boldsymbol{u}_h,\boldsymbol{p}_h),
  (\boldsymbol{v}_h,\boldsymbol{q}_h)\rangle_{\boldsymbol{X}_h^{\ast},\boldsymbol{X}_h}
  = a_h((\boldsymbol{u}_h,\boldsymbol{p}_h),
  (\boldsymbol{v}_h,\boldsymbol{q}_h))$. The operator $A$ can be
  ordered to have the following structure:
  \begin{equation*}
    A =
    \begin{bmatrix}
      A_{11} & A_{21}^T
      \\
      A_{21} & A_{22}
    \end{bmatrix}
    =
    \left[
      \begin{array}{c c | c c}
        C_{11} & B_{11}^T & C_{21}^T & B_{21}^T
        \\
        B_{11} & 0 & 0 & 0
        \\
        \hline
        C_{21} & 0 & C_{22} & 0
        \\
        B_{21} & 0 & 0 & 0
      \end{array}
    \right].
  \end{equation*}
  \textbf{Step (iii).} We define the preconditioner
  $P: \boldsymbol{X}_h \to \boldsymbol{X}_h^{\ast}$ and the operators
  $P^u: \boldsymbol{V}_h \to \boldsymbol{V}_h^{\ast}$ and
  $P^p: \boldsymbol{Q}_h \to \boldsymbol{Q}_h^{\ast}$ such that
  \begin{equation*}
    \begin{split}
      \langle P(\boldsymbol{u}_h,\boldsymbol{p}_h), (\boldsymbol{v}_h,\boldsymbol{q}_h)\rangle_{\boldsymbol{X}_h^{\ast},\boldsymbol{X}_h}
      &=
      ((\boldsymbol{u}_h,\boldsymbol{p}_h), (\boldsymbol{v}_h,\boldsymbol{q}_h))_{\boldsymbol{X}_h}
      \\
      &=
      \langle P^u\boldsymbol{u}_h, \boldsymbol{v}_h\rangle_{\boldsymbol{V}_h^{\ast},\boldsymbol{V}_h}
      + \langle P^p\boldsymbol{p}_h, \boldsymbol{q}_h\rangle_{\boldsymbol{Q}_h^{\ast},\boldsymbol{Q}_h}.
    \end{split}
  \end{equation*}
  We note that $P$ has the following structure:
  \begin{equation*}
    P = 
    \begin{bmatrix}
      P^u & 0 \\
      0 & P^p
    \end{bmatrix}
    =
    \begin{bmatrix}
      P^u_{11} & (P^u_{21})^T & 0 & 0
      \\
      P^u_{21} & P^u_{22} & 0 & 0
      \\
      0 & 0 & P^p_{11} & 0
      \\
      0 & 0 & 0 & P^p_{22}
    \end{bmatrix},
  \end{equation*}
  with $P_{11}^u : V_h \to V_h^{\ast}$,
  $P_{21}^u : V_h \to \bar{V}_h^{\ast}$,
  $P_{22}^u : \bar{V}_h \to \bar{V}_h$, $P^p_{11}:Q_h \to Q_h$, and
  $P^p_{22}:\bar{Q}_h \to \bar{Q}_h$. We associate with $P$ the
  operator
  $S_P : (\bar{V}_h \times \bar{Q}_h) \to (\bar{V}_h^{\ast} \times
  \bar{Q}_h^{\ast})$ defined by
  \begin{equation}
  \label{eq:StokesPreconditioner}
    S_P =
    \begin{bmatrix}
      S_{P^u} & 0
      \\
      0 & S_{P^p}
    \end{bmatrix},
  \end{equation}
  with $S_{P^u} := P_{22}^u - P_{21}^u(P_{11}^u)^{-1}(P_{21}^u)^T$ and
  $S_{P^p} := P_{22}^p$. Similar to \cref{eq:AuhvhgeSA}, we have
  \begin{equation}
    \label{eq:AuhvhgeSAStokes_a}
    \begin{split}
      \langle P^u\boldsymbol{v}_h, \boldsymbol{v}_h \rangle_{\boldsymbol{V}_h^{\ast},\boldsymbol{V}_h}
      =& \langle P_{11}^u(v_h + (P_{11}^u)^{-1}(P_{21}^u)^T\bar{v}_h), v_h + (P_{11}^u)^{-1}(P_{21}^u)^{T}\bar{v}_h \rangle_{V_h^{\ast},V_h}
      \\
      &+ \langle S_{P^u}\bar{v}_h, \bar{v}_h \rangle_{\bar{V}_h^{\ast},\bar{V}_h}
      \quad \forall \boldsymbol{v}_h \in \boldsymbol{V}_h.
    \end{split}
  \end{equation}
  \textbf{Step (iv).} Using \cref{eq:chboundStokes,eq:CresultS} we
  find
  \begin{multline}
    \label{eq:PuvhvhScvbar}
    \langle P^u\boldsymbol{v}_h, \boldsymbol{v}_h\rangle_{\boldsymbol{V}_h^{\ast},\boldsymbol{V}_h}
    = \tnorm{\boldsymbol{v}_h}_{v,S}^2
    \ge c_1^{-1} c_h(\boldsymbol{v}_h,\boldsymbol{v}_h)
    \\
    = c_1^{-1} \langle C\boldsymbol{v}_h, \boldsymbol{v}_h \rangle_{\boldsymbol{V}_h^{\ast},\boldsymbol{V}_h}
    \ge c_1^{-1} \langle S_C \bar{v}_h, \bar{v}_h \rangle_{\bar{V}_h^{\ast}, \bar{V}_h}
    \quad \forall \boldsymbol{v}_h \in \boldsymbol{V}_h.
  \end{multline}
  Therefore, choosing $v_h = -(P_{11}^u)^{-1}(P_{21}^u)^T\bar{v}_h$ in
  \cref{eq:PuvhvhScvbar} we find from \cref{eq:AuhvhgeSAStokes_a} that
  \begin{equation}
    \label{eq:SPuSCbound}
    \langle S_{P^u}\bar{v}_h, \bar{v}_h \rangle_{\bar{V}_h^{\ast},\bar{V}_h} \ge
    c_1^{-1} \langle S_C \bar{v}_h, \bar{v}_h \rangle_{\bar{V}_h^{\ast}, \bar{V}_h}
    \quad \forall \bar{v}_h \in \bar{V}_h.
  \end{equation}
  Furthermore, note that 
  \begin{equation}
    \label{eq:SPpqbqbh}
    \sum_{K \in \mathcal{T}_h} \nu^{-1}\eta^{-1}h_K\norm[0]{\bar{q}_h}_{\partial K}^2
    = \langle S_{P^p} \bar{q}_h, \bar{q}_h \rangle_{\bar{Q}_h^{\ast}, \bar{Q}_h}.
  \end{equation}
  \textbf{Step (v).} We remark that \cref{lem:mainLemmaStep2Stokes}
  in operator form reads as
  \begin{multline}
    \label{eq:leminopformStokes}
    \tnorm{(-A_{11}^{-1}A_{21}^T(\bar{v}_h,\bar{q}_h), (\bar{v}_h,\bar{q}_h))}_{\boldsymbol{X}_h}^2
    \\
    \le
    c_s \eta\del[2]{ \eta \langle S_C\bar{v}_h, \bar{v}\rangle_{V_h^{\ast},V_h}
      + \nu^{-1}\eta^{-1}\norm[0]{h_K^{1/2}\bar{q}_h}_{\partial \mathcal{T}_h}^2 }.
  \end{multline}
  Combining \cref{eq:SPuSCbound,eq:SPpqbqbh,eq:leminopformStokes},
  \begin{equation*}
    \tnorm{(-A_{11}^{-1}A_{21}^T(\bar{v}_h,\bar{q}_h), (\bar{v}_h,\bar{q}_h))}_{\boldsymbol{X}_h}^2
    \le
    c \eta^2\del[2]{ \langle S_{P^u}\bar{v}_h, \bar{v}_h \rangle_{\bar{V}_h^{\ast},\bar{V}_h}
      + \langle S_{P^p} \bar{q}_h, \bar{q}_h \rangle_{\bar{Q}_h^{\ast}, \bar{Q}_h} }.
  \end{equation*}
  Since
  $\langle S_P(\bar{v}_h,\bar{q}_h), (\bar{v}_h,\bar{q}_h)
  \rangle_{\bar{X}_h^{\ast},\bar{X}_h} =\langle S_{P^u}\bar{v}_h,
  \bar{v}_h \rangle_{\bar{V}_h^{\ast},\bar{V}_h} + \langle S_{P^p}
  \bar{q}_h, \bar{q}_h \rangle_{\bar{Q}_h^{\ast}, \bar{Q}_h} =
  (\bar{v}_h, \bar{v}_h)_{\bar{V}_h} +
  (\bar{q}_h,\bar{q}_h)_{\bar{Q}_h} =
  \norm[0]{(\bar{v}_h,\bar{q}_h)}_{\bar{X}_h}$, with
  $\bar{X}_h := \bar{V}_h \times \bar{Q}_h$, the result follows.
\end{proof}

That $S_P^{-1}$ is a parameter-robust preconditioner for the reduced
problem of the Stokes discretization \cref{eq:hdg-scheme-stokes} is
now a consequence of \cref{thm:SpparamrobSa}.

\begin{remark}
  The preconditioner defined here is different from the
  preconditioners studied in
  \cite{rhebergen2022preconditioning}. There, the preconditioners at
  the condensed level are
  \begin{equation*}
   (S_P^1)^{-1} =
   \begin{bmatrix}
     \bar{C}^{-1} & 0
     \\
     0 & (P_{22}^p)^{-1}
   \end{bmatrix}
   \qquad\text{and}\qquad
   (S_P^2)^{-1} =
   \begin{bmatrix}
     \bar{C}^{-1} & 0
     \\
     0 & (B_{21}C_{11}^{-1}B_{21}^T)^{-1}
   \end{bmatrix},
  \end{equation*}
  where $\bar{C} = -C_{21}\mathcal{P}C_{11}^{-1}C_{21}^T + C_{22}$
  with
  $\mathcal{P} = I -
  C_{11}^{-1}B_{11}^T(B_{11}C_{11}^{-1}B_{11}^T)^{-1}B_{11}$.
\end{remark}

\section{Numerical Examples}
\label{s:num}

In this section we verify that the reduced preconditioners $S_P^{-1}$,
obtained from the full preconditioners $P^{-1}$ (for the Darcy problem
defined by the inner product \cref{eq:darcyInnerProd} and for the
Stokes problem by the inner product \cref{eq:innerproductXhStokes})
are parameter-robust for the statically condensed discretization of
the Darcy problem and the Stokes problem. The discretizations and
preconditioners were implemented in Netgen/NGSolve
\cite{schoberl1997netgen,schoberl2014c++}.

We consider simulations with both exact preconditioners (in which we
use direct solvers to invert the relevant blocks in the
preconditioner) and inexact preconditioners. For the inexact
preconditioners we use a minor modification of the auxiliary space
preconditioner (ASP) of \cite{fu2021uniform,fu2023uniform} (taking
into account the problem parameters in the auxiliary problem and
modifying the `prolongation' operators specifically for the hybridized
methods used here) in combination with classical algebraic multigrid
via BoomerAMG \cite{yang2002boomeramg} through PETSc
\cite{petsc-user-ref,petsc-web-page}.

The preconditioners are combined with the conjugate gradient (CG)
method for the reduced form of the Darcy problem and minimum residual
(MINRES) method for the reduced form of the Stokes problem. For both
the Darcy and Stokes problems we choose the penalty parameter to be
$\eta = 4k^2$ in 2d and $\eta = 6k^2$ in 3d. For all examples we
choose $k=2$.

\subsection{The Darcy problem: numerical results}
\label{ss:darcyNumRes}

\subsubsection{A manufactured solution}
\label{sss:manufacturedsol}

We first consider a manufactured solution in 2d and 3d on the domain
$\Omega = (0,1)^d$, $d=2,3$. In 2d the source and boundary terms are
set such that $p = \cos(\pi x_1)\sin(\pi x_2)$ and in 3d such that
$p = \cos(\pi x_1)\sin(\pi x_2)\cos(\pi x_3)$. In
\cref{table:precon-darcy-manu} we observe that the number of
iterations for CG to reach a relative preconditioned residual
tolerance of $10^{-10}$ for $\gamma \in \cbr[0]{10^{-4},1,10^4}$ and
$\xi \in \cbr[0]{10^{-6},1}$ is independent of the discretization and
model parameters.

\begin{table}[tbp]
  \centering
  \resizebox{\textwidth}{!}{%
    \begin{tabular}{c|c|c|c|c|c|c|c|c|c}
      \hline
      \text{Cells 2D/3D} & \multicolumn{3}{|c|}{$138/455$} & \multicolumn{3}{|c|}{$610/3804$} & \multicolumn{3}{|c}{$2400/24892$} \\
      \hline
      $\gamma$ & $10^{4}$ & $10^{0}$ & $10^{-4}$ & $10^{4}$ & $10^{0}$ & $10^{-4}$ &  $10^{4}$ & $10^{0}$ & $10^{-4}$\\
      \hline\hline      
      \multicolumn{10}{c}{$\xi = 10^0$}\\
      \hline
      2D & 29 (33) & 32 (40) & 32 (40) & 31 (34) & 31 (41) & 31 (41)  & 33 (36) & 31 (42) & 31 (42) \\
      \hline
      3D & 35 (37) & 45 (55) & 45 (55) & 42 (48) & 48 (60) & 48 (60) &  50 (54) & 52 (65) & 52 (65) \\
      \hline
      \multicolumn{10}{c}{$\xi = 10^{-6}$}\\
      \hline
      2D & 28 (28) & 28 (28) & 32 (39) & 28 (28) & 28 (28) & 32 (41)  & 30 (30) & 29 (30) & 32 (42) \\
      \hline
      3D & 34 (34) & 34 (34) & 44 (52) & 39 (39) & 39 (39) & 48 (59) &  43 (43) & 43 (43) & 52 (64) \\
      \hline
    \end{tabular}}
  \caption{The Darcy problem with manufactured solution (see
    \cref{sss:manufacturedsol}). Number of iterations required for CG
    to reach a relative preconditioned residual tolerance of
    $10^{-10}$ for different values of $\xi$ and $\gamma$. In the 2d
    and 3d columns, numbers shown in parantheses are obtained using
    the inexact preconditioner while the remaining numbers are
    obtained using the exact preconditioner.}
  \label{table:precon-darcy-manu}
\end{table}

\subsubsection{Piecewise constant reaction parameter}
\label{sss:piecewiseconstant}

Next, we consider a more challenging problem, in which we consider a
piecewise constant reaction coefficient $\gamma$. In 2d we define
$\Omega = (0,1)^2$ with subdomains $\Omega_1 = (0.3,0.7)^2$ and
$\Omega_2 = \Omega \backslash \Omega_1$. In 3d we define
$\Omega = (0,1)^3$ with subdomains $\Omega_1 = (0.3,0.7)^3$ and
$\Omega_2 = \Omega \backslash \Omega_1$. We set
$\xi = 1 + \sum_{i=1}^d(x_i-0.5)^2$ and $\gamma=\gamma_i$ on
$\Omega_i$ with $(\gamma_1,\gamma_2) = (1,10^{4})$. We choose $f=1$
and homogeneous Dirichlet boundary conditions. In
\cref{table:darcy-pw-constant} we observe robustness for the exact and
inexact preconditioners in two and three dimensions.

\begin{table}[tbp]
  \centering
  \begin{tabular}{c|c|c|c|c|c}
    \hline
    Cells (2d) & 138 & 608 & 2382 & 9512 & 37938 \\
    \hline
    Iterations (2d) & 30 (37) & 31 (42) & 31 (41) & 31 (41) & 30 (41) \\
    \hline
    \hline
    Cells (3d) & 53 & 455 & 3804 & 24892 & 194816 \\
    \hline
    Iterations (3d) & 34 (34) & 35 (36) & 43 (54) & 47 (59) & 49 (63) \\
    \hline
  \end{tabular}
  \caption{The Darcy problem with piecewise constant reaction
    parameter (see \cref{sss:piecewiseconstant}). Number of iterations
    required for CG to reach a relative preconditioned residual
    tolerance of $10^{-10}$. The numbers shown in parantheses are
    obtained using the inexact preconditioner while the remaining
    numbers are obtained using the exact preconditioner.}
  \label{table:darcy-pw-constant}
\end{table}

\subsubsection{When the conditions in \cref{thm:SpparamrobSa} are not satisfied}
\label{sss:notsatisfied}

In this section we demonstate that a robust preconditioner for the
full system does not imply that the reduced preconditioner is also
robust for the reduced problem. For this we consider the
preconditioner associated with the inner product
\cref{eq:darcyInnerProdalt}. As test case we consider the 2d
manufactured solution from \cref{sss:manufacturedsol} with
$\kappa = \gamma = 1$. We use MINRES to solve the full system with
preconditioner $(P')^{-1}$ and CG to solve the reduced problem with
preconditioner $S_{P'}^{-1}$. In \cref{tab:counter-example} we list
the number of iterations to reach a relative preconditioned residual
tolerance of $10^{-10}$. We observe that while the preconditioner is
robust for the full system, its reduced version is not robust for the
system obtained after static condensation.

\begin{table}[tbp]
  \centering
  \begin{tabular}{c|c|c|c|c|c|c}
    \hline
    Cells & 138 & 608 & 2382 & 9512 & 37938 & 151526 \\
    \hline
    MINRES with $(P')^{-1}$  & 14 & 14 & 13 & 12 & 11 & 12 \\
    \hline
    CG with $S_{P'}^{-1}$ & 129 & 236 & 452 & 758 & $>999$ & $>999$ \\
    \hline
  \end{tabular}
  \caption{The test case from \cref{sss:notsatisfied}. Number of
    iterations, using MINRES with $(P')^{-1}$ for the full system and
    CG with $S_{P'}^{-1}$ for the reduced system, to reach a relative
    preconditioned residual tolerance of $10^{-10}$.}
  \label{tab:counter-example}
\end{table}

\subsection{The Stokes problem: numerical results}
\label{ss:stokesNumRes}

In addition to the preconditioner studied in \cref{ss:stokes}, we also
consider two norm-equivalent preconditioners (see
\cref{rem:SPhinnerprod}). In particular, we consider the
preconditioners $P_{\zeta}^{-1}$ and $\widehat{P}_{\zeta}^{-1}$
associated with the inner product obtained by replacing
$(\boldsymbol{u}_h, \boldsymbol{v}_h)_{v,S}$ in
\cref{eq:innerproductXhStokes} by
\begin{equation*}
  \begin{split}
    (\boldsymbol{u}_h, \boldsymbol{v}_h)_{v,S,\zeta}
    &:= (\boldsymbol{u}_h, \boldsymbol{v}_h)_{v,S}
    + \zeta (\nabla \cdot u_h, \nabla \cdot v_h)_{\mathcal{T}_h},
    \\
    \widehat{c}_h(\boldsymbol{u}_h,\boldsymbol{v}_h)
    &:=
    c_h(\boldsymbol{u}_h,\boldsymbol{v}_h)
    + \zeta(\nabla \cdot u_h, \nabla \cdot v_h)_{\mathcal{T}_h},    
  \end{split}
\end{equation*}
respectively, for nonnegative parameter $\zeta$. Note that $P_0 = P$.

\subsubsection{A manufactured solution}
\label{sss:manufacturedsolStokes}

We consider a manufactured solution in the domain
$\Omega = (0,1)^d$, $d=2,3$. In 2d the source and boundary terms are
set such that
\begin{equation*}
  u =
  \begin{bmatrix}
    \sin(\pi x_1) \sin(\pi x_2) \\
    \cos(\pi x_1) \cos(\pi x_2)    
  \end{bmatrix},
  \quad
  p = \sin(\pi x_1) \cos(\pi x_2),
\end{equation*}
and in 3d the manufactured solution is given by
\begin{equation*}
  u =
  \begin{bmatrix}
    \pi \sin(\pi x_1) \cos(\pi x_2) - \pi \sin(\pi x_1) \cos(\pi x_3)  \\
    \pi \sin(\pi x_2) \cos(\pi x_3) - \pi \sin(\pi x_2) \cos(\pi x_1)  \\
    \pi \sin(\pi x_3) \cos(\pi x_1) - \pi \sin(\pi x_3) \cos(\pi x_2)
  \end{bmatrix},
  \quad
  p = \cos(\pi x_1) \sin(\pi x_2) \cos(\pi x_3).
\end{equation*}
In \cref{table:precon1-stokes} we present the number of iterations
required for preconditioned MINRES, using preconditioners
$P^{-1}=P_0^{-1}$, $\widehat{P}_0^{-1}$, and $\widehat{P}_{100}^{-1}$
to reach a relative preconditioned residual tolerance of $10^{-8}$. In
2d we observe that the exact and inexact preconditioners are
discretization and model parameter-robust. In 3d the exact
preconditioners are robust with respect to the mesh size, but the
iteration counts show some minor sensitivity regarding $\nu$. For the
inexact preconditioners in 3d we observe some discretization and model
parameter sensitivity, but variations are small.

\begin{table}[tbp]
  \centering
  \resizebox{\textwidth}{!}{%
  \begin{tabular}{c|c|c|c|c||c|c|c|c}
    \hline
    &\multicolumn{4}{|c||}{$\nu = 1$} & \multicolumn{4}{|c|}{$\nu = 1e-6$} \\
    \hline
    Cells & $P^{-1}$ & $P_{100}^{-1}$ & $\widehat{P}_0^{-1}$ & $\widehat{P}_{100}^{-1}$ & $P^{-1}$ & 
    $P_{100}^{-1}$ & $\widehat{P}_0^{-1}$ & $\widehat{P}_{100}^{-1}$\\
    \hline
    & \multicolumn{8}{c}{2d problem} \\
    \hline
    138 & 92 (121) & 58 (111) & 83 (122) & 41 (113) & 100 (133) & 63 (113) & 88 (129) & 42 (115) \\
    608 & 90 (129) & 58 (110) & 83 (130) & 41 (114) & 93 (137) & 62 (112) & 84 (134) & 41 (113) \\
    2382 & 90 (132) & 58 (112) & 84 (133) & 41 (115) & 92 (140) & 60 (113) & 83 (136) & 40 (114) \\
    9512 & 87 (134) & 58 (115) & 83 (135) & 41 (117) & 88 (143) & 59 (114) & 81 (139) & 40 (115) \\
    37938 & 86 (134) & 58 (113) & 81 (136) & 40 (117) & 88 (143) & 58 (114) & 79 (139) & 39 (115) \\
    151462 & 86 (136) & 56 (116) & 81 (133) & 40 (119) & 85 (145) & 55 (116) & 79 (141) & 37 (117) \\
    \hline
    & \multicolumn{8}{c}{3d problem} \\
    \hline
    48 & 110 (124)  & 66 (75) & 103 (118) & 47 (58) & 141 (154) & 87 (99) & 127 (141) & 60 (76)\\
    384 & 130 (149) & 69 (115) & 124 (148) & 46 (114) & 165 (187) & 91 (145) & 158 (180) & 60 (141)\\
    3072 & 134 (165) & 69 (137) & 129 (169) & 46 (144) & 165 (203) & 90 (171) & 161 (206) & 60 (175)\\
    24310 & 130 (171) & 68 (146) & 129 (176) & 45 (151) & 161 (216) & 88 (179) & 157 (220) & 58 (183)\\
    196608 & 128 (169) & 65 (142) & 125 (178) & 44 (150) & 156 (218) & 85 (177) & 151 (218) & 57 (181) \\
    \hline
  \end{tabular}
  }
  \caption{The Stokes problem with manufactured solution (see
    \cref{sss:manufacturedsolStokes}).  Number of iterations required
    for preconditioned MINRES to reach a relative preconditioned
    residual tolerance of $10^{-8}$ for different values of $\nu$. We
    compare the performance of the preconditioners $P^{-1}=P_0^{-1}$,
    $\widehat{P}_0^{-1}$, and $\widehat{P}_{100}^{-1}$ for the Stokes
    discretization \cref{eq:hdg-scheme-stokes}. The performance of the
    inexact preconditioners are shown in parentheses while the
    remaining numbers are obtained using the exact preconditioner.}
\label{table:precon1-stokes}
\end{table}

\subsubsection{Lid-driven cavity problem}
\label{sss:liddrivencavity}

We now consider the lid-driven cavity problems in 2d
($\Omega = (-1,1)^2$, $u=1-x_1^4$ on the boundary $x_2=1$ and $u=0$
elsewhere), and in 3d ($\Omega = (0,1)^3$,
$u=(1-\tau_1^4,(1-\tau_2)^4/10,0)$ with $\tau_i=2x_i-1$ on the
boundary $x_3=1$, and $u=0$ elsewhere). In
\cref{table:liddriven-stokes} we list the number of MINRES iterations
needed to reach a relative preconditioned residual tolerance of
$10^{-8}$ for varying mesh size $h$ and viscosity parameter
$\nu$. Both in 2d and 3d we observe robustness of all preconditioners
with respect to mesh size and $\nu$.

\begin{table}[tbp]
  \centering
  \resizebox{\textwidth}{!}{%
  \begin{tabular}{c|c|c|c|c||c|c|c|c}
    \hline
    &\multicolumn{4}{|c||}{$\nu = 1$} & \multicolumn{4}{|c|}{$\nu = 1e-6$} \\
    \hline
    Cells & $P^{-1}$ & $P_{100}^{-1}$ & $\widehat{P}_0^{-1}$ & $\widehat{P}_{100}^{-1}$ & $P^{-1}$ & 
    $P_{100}^{-1}$ & $\widehat{P}_0^{-1}$ & $\widehat{P}_{100}^{-1}$\\
    \hline
    & \multicolumn{8}{c}{2d problem} \\
    \hline
    138 & 101 (134) & 60 (119) & 89 (135) & 41 (122) & 101 (134) & 60 (119) & 89 (135) & 41 (121) \\
    608 & 101 (141) & 60 (122) & 89 (142) & 41 (125) & 101 (141) & 60 (122) & 89 (142) & 41 (125) \\
    2388 & 98 (144) & 58 (124) & 90 (145) & 41 (126) & 98 (144) & 58 (123) & 90 (145) & 41 (126) \\
    9538 & 97 (146) & 56 (124) & 87 (146) & 41 (126) & 97 (146) & 57 (124) & 87 (146) & 39 (126) \\
    37880 & 94 (146) & 55 (123) & 87 (146) & 39 (127) & 94 (146) & 55 (122) & 87 (146) & 39 (127) \\
    151472 & 93 (148) & 53 (124) & 83 (148) & 39 (128) & 93 (148) & 55 (124) & 83 (148) & 37 (128) \\
    \hline
    & \multicolumn{8}{c}{3d problem} \\
    \hline
    48 & 160 (167) & 91 (101) & 146 (159) & 63 (76) & 160 (167) & 91 (103) & 146 (159) & 65 (78)\\
    384 & 182 (206) & 94 (156) & 176 (204) & 63 (152) & 182 (207) & 94 (158) & 176 (204) & 62 (156)\\
    3072 & 182 (222) & 93 (187) & 178 (231) & 62 (187) & 182 (222) & 93 (186) & 178 (231) & 62 (195) \\
    24310 & 181 (232) & 91 (194) & 177 (241) & 61 (203) & 181 (232) & 91 (194) & 177 (241) & 60 (203) \\
    196608 & 177 (235) &  90 (195) & 173 (243) & 59 (202) & 177 (235) & 90 (195) & 173 (243) & 59 (202)\\
    \hline
  \end{tabular}
  }
  \caption{The lid-driven cavity problem (see
    \cref{sss:liddrivencavity}).  Number of iterations required for
    preconditioned MINRES to reach a relative preconditioned residual
    tolerance of $10^{-8}$ for different values of $\nu$. We compare
    the performance of the preconditioners $P^{-1}=P_0^{-1}$,
    $\widehat{P}_0^{-1}$, and $\widehat{P}_{100}^{-1}$ for the Stokes
    discretization \cref{eq:hdg-scheme-stokes}. The performance of the
    inexact preconditioners are shown in parentheses while the
    remaining numbers are obtained using the exact preconditioner.}
\label{table:liddriven-stokes}
\end{table}

\section{Conclusions}
\label{s:conclusions}

In this paper we presented an approach to construct parameter-robust
preconditioners for the reduced form of a hybridizable
discretization. Having defined a preconditioner for the discretization
before static condensation, we presented conditions that must be
satisfied to guarantee parameter-robustness of the reduced form of the
preconditioner for the reduced form of the discretization. We applied
this approach to determine new parameter-robust preconditioners for
the Darcy and Stokes problems. Numerical examples in two and three
dimensions verified our analysis.

\bibliographystyle{plain}
\bibliography{references}
\appendix
\section{The Darcy problem: Proofs}

The following inf-sup condition was shown in \cite[Lemma
4]{kraus2021uniformly}. There exists a uniform constant $\beta > 0$
such that
\begin{equation}
  \label{eq:inf-sup-H1}
  \inf_{\boldsymbol{0} \ne \boldsymbol{q}_h \in \boldsymbol{Q}_h^0}
  \sup_{ 0 \ne v_h \in V_h} \frac{b_h(v_h, \boldsymbol{q}_h)}{\norm[0]{v_h}_{\mathcal{T}_h} \tnorm{\boldsymbol{q}_h}_p}
  \ge \beta.
\end{equation}
This result will be used in the proofs in this appendix.

\subsection{Proof of \cref{lem:cbciDarcy}}
\label{ss:Proof_cbciDarcy}

We first prove \cref{eq:wpconditions_b}. Let
$(u_h,\boldsymbol{p}_h), (v_h,\boldsymbol{q}_h) \in
\boldsymbol{X}_h$. Applying the Cauchy--Schwarz inequality we find:
\begin{equation}
\label{eq:kappainvuvgampq}
    \xi^{-1}(u_h,v_h)_{\mathcal{T}_h} \le \tnorm{u_h}_{v,D}\tnorm{v_h}_{v,D},
    \quad
    \gamma(p_h,q_h)_{\mathcal{T}_h} \le \tnorm{\boldsymbol{p}_h}_{q,D}\tnorm{\boldsymbol{q}_h}_{q,D}.    
\end{equation}
Next, applying integration by parts, the Cauchy--Schwarz inequality,
and a discrete trace inequality \cite[Lemma 1.46]{di2011mathematical},
we find:
\begin{equation*}
  \begin{split}
    b_h(&u_h,\boldsymbol{q}_h)
    = (u_h, \nabla q_h)_{\mathcal{T}_h} + \langle \bar{q}_h - q_h, u_h \cdot n \rangle_{\partial \mathcal{T}_h}
    \\
    &\le \xi^{-1/2} \norm[0]{u_h}_{\mathcal{T}_h} \xi^{1/2} \norm[0]{\nabla q_h}_{\mathcal{T}_h}
    + C \eta^{-1/2} \norm[0]{(\xi \eta h_K^{-1})^{1/2} (\bar{q}_h - q_h)}_{\partial \mathcal{T}_h}
    \xi^{-1/2} \norm[0]{u_h}_{\mathcal{T}_h},
  \end{split}
\end{equation*}
from which we find
\begin{equation}
  \label{eq:boundbhuhqh}
  b_h(u_h,\boldsymbol{q}_h) \le C(1+\eta^{-1})^{1/2} \tnorm{u_h}_{v,D}\tnorm{\boldsymbol{q}_h}_{q,D}.
\end{equation}
A similar estimate holds for $b_h(v_h,\boldsymbol{p}_h)$. Combining
\cref{eq:kappainvuvgampq,eq:boundbhuhqh}, and since $\eta > 1$, we
find that the constant $c_b$ in \cref{eq:wpconditions_b} is uniform.

We now prove \cref{eq:wpconditions_i}. We note that if for all
$(u_h,\boldsymbol{p}_h) \in \boldsymbol{X}_h$ we can find
$(v_h,\boldsymbol{q}_h) \in \boldsymbol{X}_h$ depending on
$(u_h,\boldsymbol{p}_h)$ such that
\begin{equation}
  \label{eq:darcy22b}
    \tnorm{(v_h,\boldsymbol{q}_h)}_{\boldsymbol{X}_h} \le c_1 \tnorm{(u_h,\boldsymbol{p}_h)}_{\boldsymbol{X}_h},
    \quad
    a_h((u_h, \boldsymbol{p}_h), (v_h, \boldsymbol{q}_h)) \ge c_2 \tnorm{(u_h,\boldsymbol{p}_h)}_{\boldsymbol{X}_h}^2,
\end{equation}
for positive uniform constants $c_1,c_2$, then
\cref{eq:wpconditions_i} holds with $c_i>0$ a uniform constant. Thanks
to \cref{eq:inf-sup-H1}, given a
$\boldsymbol{p}_h \in \boldsymbol{Q}_h^0$ there exists a
$\tilde{u}_h \in V_h$ such that
\begin{equation}
  \label{eq:bhutphutc1ph}
  b_h(\tilde{u}_h,\boldsymbol{p}_h) = \tnorm{\boldsymbol{p}_h}_p^2
  \qquad \text{and} \qquad
  \norm[0]{\tilde{u}_h}_{\mathcal{T}_h} \le c \tnorm{\boldsymbol{p}_h}_p,
\end{equation}
where $c>0$ is a constant only depending on $\beta$. Define
$v_h := u_h + \delta \xi \tilde{u}_h$ and
$\boldsymbol{q}_h := \boldsymbol{p}_h$, where $\delta>0$ is a constant
that will be determined below. Applying \cref{eq:bhutphutc1ph} we find:
\begin{equation}
  \label{eq:boundvhvD}
    \tnorm{v_h}_{v,D}^2
    \le 2 \tnorm{u_h}_{v,D}^2 + 2\delta^2\xi \norm[0]{\tilde{u}_h}_{\mathcal{T}_h}^2
    \le 2 \tnorm{u_h}_{v,D}^2 + 2\delta^2c^2 \tnorm{\boldsymbol{p}_h}_{q,D}^2.
\end{equation}
\Cref{eq:boundvhvD} together with
$\boldsymbol{q}_h := \boldsymbol{p}_h$ now implies the first inequality in
\cref{eq:darcy22b}. Next, applying \cref{eq:bhutphutc1ph} and the
Cauchy--Schwarz and Young's inequalities, we find:
\begin{align*}
  &a_h((u_h, \boldsymbol{p}_h), (v_h, \boldsymbol{q}_h))
    = \xi^{-1} \norm[0]{u_h}_{\mathcal{T}_h}^2
    + \delta (u_h, \tilde{u}_h)_{\mathcal{T}_h}
    + \delta \xi b_h(\tilde{u}_h, \boldsymbol{p}_h)
    + \gamma \norm[0]{r_h}_{\mathcal{T}_h}^2
  \\
  &\quad \ge \xi^{-1} \norm[0]{u_h}_{\mathcal{T}_h}^2
    - \frac{\delta c^2 \xi^{-1}}{2} \norm[0]{u_h}_{\mathcal{T}_h}^2
    - \frac{\delta \xi}{2} \tnorm{\boldsymbol{p}_h}_p^2
    + \delta \xi \tnorm{\boldsymbol{p}_h}_p^2
    + \gamma \norm[0]{r_h}_{\mathcal{T}_h}^2.
\end{align*}
Choosing $\delta = 1/c^2$ we obtain the second inequality in \cref{eq:darcy22b}.

\subsection{Proof of \cref{lem:mainLemmaStep2Darcy}}
\label{ss:Proof_mainLemmaStep2Darcy}

Before proving \cref{lem:mainLemmaStep2Darcy} we require the following
preliminary results. Recall the definition of $m_K$ in
\cref{ss:prelim}.

\begin{lemma}
  \label{lem:previous-condensed-estimate-Darcy}
  Let $l_u(\cdot)$ and $l_p(\cdot)$ be as defined in
  \cref{lem:weak-reduced-darcy}. Given $\bar{q}_h \in \bar{Q}_h^0$,
  there exists a uniform constant $c>0$ such that
  \begin{multline*}
    \xi^{-1} \norm[0]{l_u(\bar{q}_h)}_{K}^2
    + \gamma \norm[0]{l_p(\bar{q}_h)}_{K}^2
    \\
    \le c \del[1]{ \xi h_K^{-1} \norm[0]{\bar{q}_h - m_K(\bar{q}_h) }_{\partial K}^2
    + \min \cbr[0]{\gamma, h_K^{-2}\xi} \norm[0]{m_K(\bar{q}_h)}_{K}^2 }.
  \end{multline*}
\end{lemma}
\begin{proof}
  For this proof we adapt results from \cite[Section
  6]{cockburn2005error}.
  \\
  \textbf{Step (i).} Given $K \in \mathcal{T}_h$, we define
    $V^0(K) := \cbr[1]{ v \in V(K):\, \nabla\cdot v = 0 }$
  where $V(K) := [\mathbb{P}_k(K)]^d$. We define $J$ as the
  orthogonal projection into $V^0(K)$: for given $v_h \in V_h$,
  \begin{equation}
    \label{eq:operator-J}
    \xi^{-1} (J v_h, w_h)_K
    = \xi^{-1} (v_h, w_h)_K \qquad \forall w_h \in V^0(K).
  \end{equation}
  Given $\bar{q}_h \in \bar{Q}_h^0$, by the orthogonality of $J$, we
  have
  \begin{equation}
    \label{eq:aux1-previous-condensed-estimate-Darcy}
    \xi^{-1} \norm[0]{l_u(\bar{q}_h)}_K^2
    = \xi^{-1} \norm[0]{J l_u(\bar{q}_h)}_K^2
    + \xi^{-1} \norm[0]{l_u(\bar{q}_h) - J l_u(\bar{q}_h)}_K^2.
  \end{equation}
  \textbf{Step (ii).}  To bound the first term on the right hand side
  of \cref{eq:aux1-previous-condensed-estimate-Darcy}, we first note
  that the local problem in \cref{def:locSolverDarcy} with $q_h = 0$
  gives
  \begin{equation*}
      \xi^{-1}(l_u(\bar{q}_h),v_h)_K - (l_p(\bar{q}_h),\nabla \cdot v_h)_K
      = -\langle \bar{q}_h, v_h \cdot n \rangle_{\partial K} 
      \quad
      \forall v_h \in V(K).
  \end{equation*}
  Choosing $v_h = Jl_u(\bar{q}_h) \in V^0(K)$, using that
  $\nabla \cdot J l_u(\bar{q}_h) = 0$ and \cref{eq:operator-J}, we
  obtain
  \begin{equation*}
    \xi^{-1}\norm[0]{J l_u(\bar{q}_h)}_K^2
    = - \langle \bar{q}_h, J l_u(\bar{q}_h) \cdot n\rangle_{\partial K}.
  \end{equation*}
  Then, by the identity
  \begin{equation*}
    \langle m_K(\bar{q}_h), J l_u(\bar{q}_h)\cdot n\rangle_{\partial
      K}=(\nabla m_K(\bar{q}_h), Jl_u(\bar{q}_h))_K + (m_K(\bar{q}_h),
    \nabla \cdot J l_u(\bar{q}_h))_K = 0,
  \end{equation*}
  the Cauchy--Schwarz inequality, and a discrete trace inequality
  \cite[Lemma 1.46]{di2011mathematical},
  \begin{align*}
    \xi^{-1} \norm[0]{J l_u(\bar{q}_h)}_K^2
    & = - \langle \bar{q}_h - m_K(\bar{q}_h), J l_u(\bar{q}_h)\cdot n\rangle_{\partial K}
    \\
    & \le C \xi^{1/2} h_K^{-1/2} \norm[0]{\bar{q}_h - m_K(\bar{q}_h)}_{\partial K} \xi^{-1/2} \norm[0]{J l_u(\bar{q}_h)}_K, 
  \end{align*}
  from which we find
  \begin{equation}
    \label{eq:aux2-previous-condensed-estimate-Darcy}
    \xi^{-1} \norm[0]{J l_u(\bar{q}_h)}_K^2
    \le C \xi h_K^{-1} \norm[0]{\bar{q}_h - m_K(\bar{q}_h)}_{\partial K}^2.
  \end{equation}
  \textbf{Step (iii).} To bound the second term of
  \cref{eq:aux1-previous-condensed-estimate-Darcy}, we define
  $z^0 := l_u(m_K(\bar{q}_h))$,
  $z' := l_u(\bar{q}_h - m_K(\bar{q}_h))$,
  $p^0 := l_p(m_K(\bar{q}_h))$,
  $p' := l_p(\bar{q}_h - m_K(\bar{q}_h))$. We note that
  $z^0 = u_h^L(m_K(\bar{q}_h), 0)$ and
  $p^0 = p_h^L(m_K(\bar{q}_h), 0)$, see
  \cref{def:locSolverDarcy}. Therefore, choosing $v_h = (I-J)z^0$ and
  $q_h = p^0$ in \cref{eq3:localsolver-darcy}, using the definition of
  $J$, and integration by parts, we find:
  \begin{equation}
    \label{eq:IminJz0p0}
    \xi^{-1} \norm[0]{(I - J)z^0}_K^2
    + \gamma \norm[0]{p^0}_K^2 
    = - \langle m_K(\bar{q}_h), (I - J)z^0 \cdot n \rangle_{\partial K}
    = - (m_K(\bar{q}_h), \nabla \cdot (I - J) z^0)_K.
  \end{equation}
  Noting that
  $(m_K(\bar{q}_h), \nabla \cdot (I - J) z^0)_K = (m_K(\bar{q}_h),
  \nabla \cdot z^0)_K$ and that $\nabla \cdot z^0 = - \gamma p^0$, and
  using Young's inequality, we find from \cref{eq:IminJz0p0}
  \begin{equation}
    \label{eq:aux3-previous-condensed-estimate-Darcy}
    \xi^{-1} \norm[0]{(I - J)z^0}_K^2
    + \gamma \norm[0]{p^0}_K^2 
    \le C \gamma \norm[0]{m_K(\bar{q}_h)}_{K}^2 .
  \end{equation}
  Alternatively, using an inverse inequality, we find from
  \cref{eq:IminJz0p0}
  \begin{equation*}
    \begin{split}
      \xi^{-1} \norm[0]{(I - J)z^0}_K^2
      + \gamma \norm[0]{p^0}_K^2 
      &= - (m_K(\bar{q}_h), \nabla \cdot (I - J) z^0)_K
      \\
      &\le C h_K^{-1} \norm[0]{m_K(\bar{q}_h)}_K \norm[0]{(I-J)z^0}_K.        
    \end{split}
  \end{equation*}
  By Young's inequality we then obtain:
  \begin{equation}
    \label{eq:aux3b-previous-condensed-estimate-Darcy}
    \xi^{-1} \norm[0]{(I - J)z^0}_K^2
    + \gamma \norm[0]{p^0}_K^2 
    \le C h_K^{-2} \xi \norm[0]{m_K(\bar{q}_h)}_K^2.
  \end{equation}
  \textbf{Step (iv).} We note that
  $z' = u_h^L(\bar{q}_h - m_K(\bar{q}_h), 0)$ and
  $p' = p_h^L(\bar{q}_h - m_K(\bar{q}_h), 0)$. Therefore, choosing
  $v_h = (I-J)z'$ and $q_h = p'$ in \cref{eq3:localsolver-darcy},
  applying the Cauchy--Schwarz inequality, a discrete trace inequality
  and Young's inequality, we find
  \begin{equation}
    \label{eq:aux4-previous-condensed-estimate-Darcy}
    \xi^{-1} \norm[0]{(I - J)z'}_K^2
    + \gamma \norm[0]{p'}_K^2 
    \le C \xi h_K^{-1} \norm[0]{\bar{q}_h - m_K(\bar{q}_h)}_{\partial K}^2.
  \end{equation}
  \textbf{Step (v).} Since $l_u(\bar{q}_h) = z^0 + z'$ and
  $l_p(\bar{q}_h) = p^0 + p'$, we find from
  \cref{eq:aux1-previous-condensed-estimate-Darcy,eq:aux2-previous-condensed-estimate-Darcy,eq:aux3-previous-condensed-estimate-Darcy,eq:aux3b-previous-condensed-estimate-Darcy,eq:aux4-previous-condensed-estimate-Darcy},
  \begin{align*}
    &\xi^{-1}\norm[0]{l_u(\bar{q}_h)}_K^2
      + \gamma \norm[0]{l_p(\bar{q}_h)}_K^2
      = \xi^{-1} \norm[0]{J l_u(\bar{q}_h)}_K^2
      + \xi^{-1} \norm[0]{(I - J)l_u(\bar{q}_h)}_K^2
      + \gamma \norm[0]{l_p(\bar{q}_h)}_K^2
    \\
    & \quad 
      \le C \del[1]{ \xi^{-1} \norm[0]{J l_u(\bar{q}_h)}_K^2
      + \xi^{-1} \norm[0]{(I - J)z^0}_K^2
      + \xi^{-1} \norm[0]{(I - J)z'}_K^2
      + \gamma \norm[0]{p^0}_K^2
      + \gamma \norm[0]{p'}_K^2 }
    \\
    & \quad 
      \le C \del[1]{ \xi h_K^{-1} \norm[0]{\bar{q}_h - m_K(\bar{q}_h)}_{\partial K}^2
      + \min \cbr[0]{\gamma, h_K^{-2}\xi } \norm[0]{m_K(\bar{q}_h)}_K^2 },
  \end{align*}
  proving the result.
\end{proof}

\begin{lemma}
  \label{lem:estimate-m_K}
  Let $\bar{q}_h \in \bar{Q}_h^0$ be given and let
  $\tilde{l}_p(\bar{q}_h)$ be as defined in
  \cref{lem:weak-reduced-darcy-aux}. There exists a uniform constant
  $c > 0$ such that
  \begin{equation}
    \label{eq:estimate-m_K}
    \min\cbr[0]{\gamma, h_K^{-2}\xi} \norm[0]{m_K(\bar{q}_h)}_{ K}^2
    \le c \del[1]{ \gamma \norm[0]{\tilde{l}_p(\bar{q}_h)}_K^2
      + \xi \eta h^{-1}_K \norm[0]{\tilde{l}_p(\bar{q}_h) - \bar{q}_h}_{\partial K}^2}.
  \end{equation}
\end{lemma}
\begin{proof}
  We define $\theta := \min\cbr[0]{\gamma, h_K^{-2} \xi}$
  and
  $M_K(\tilde{l}_p(\bar{q}_h)) :=
  \frac{1}{|K|}(\tilde{l}_p(\bar{q}_h),1)_K$. Using the triangle inequality
  and that
  $\norm[0]{a}_K^2 = a^2|K| = \frac{|K|}{|\partial
    K|}\norm[0]{a}_{\partial K}^2$ for a constant $a$, we find
  \begin{equation}
    \label{eq:MKmkterm1}
    \begin{split}
      \theta \norm[0]{m_K(\bar{q}_h)}_{K}^2
      &\le
      2 \theta \norm[0]{M_K(\tilde{l}_p(\bar{q}_h))- m_K(\bar{q}_h)}_K^2
      + 2 \theta \norm[0]{M_K(\tilde{l}_p(\bar{q}_h))}_K^2
      \\
      &= 2 \theta \norm[0]{m_K(M_K(\tilde{l}_p(\bar{q}_h)) - \bar{q}_h)}_K^2
        + 2 \theta \norm[0]{M_K(\tilde{l}_p(\bar{q}_h))}_K^2
      \\
      &\le 2 \frac{|K|}{|\partial K|} \theta \norm[0]{m_K(M_K(\tilde{l}_p(\bar{q}_h)) - \bar{q}_h)}_{\partial K}^2
        + 2 \theta \norm[0]{\tilde{l}_p(\bar{q}_h)}_K^2,
    \end{split}
  \end{equation}
  where for the last inequality we used H\"older's inequality to bound
  the last term on the right hand side. Since
  $|K|/|\partial K| \le c h_K$ we then find
  \begin{equation}
    \label{eq:MKmkterm2}
    \begin{split}
      \theta \norm[0]{m_K(\bar{q}_h)}_{K}^2
      &\le C\del[2]{h_K \theta \norm[0]{m_K(M_K(\tilde{l}_p(\bar{q}_h)) - \bar{q}_h)}_{\partial K}^2
        + 2 \theta \norm[0]{\tilde{l}_p(\bar{q}_h)}_K^2}
      \\
      &\le C\del[2]{h_K \theta \norm[0]{M_K(\tilde{l}_p(\bar{q}_h)) - \bar{q}_h}_{\partial K}^2
        + \theta \norm[0]{\tilde{l}_p(\bar{q}_h)}_K^2},
    \end{split}
  \end{equation}
  where for the last inequality we again used H\"older's inequality
  (see, e.g., \cite[eq. (10.6.11)]{brenner2008mathematical}). Next,
  applying the triangle inequality, a discrete trace inequality
  $\norm[0]{q_h}_{\partial K} \le c_th_K^{-1/2}\norm[0]{q_h}_K$
  (\cite[Lemma 1.46]{di2011mathematical}), and the inequality
  $\norm[0]{M_K(q) - q}_K \le \norm[0]{q}_K$ by the Pythagorean
  theorem, we find
  \begin{align*}
    \theta \norm[0]{m_K(\bar{q}_h)}_{K}^2
    &\le C \big( h_K \theta \norm[0]{M_K(\tilde{l}_p(\bar{q}_h)) - \tilde{l}_p(\bar{q}_h)}_{\partial K}^2
      + h_K \theta \norm[0]{\tilde{l}_p(\bar{q}_h) - \bar{q}_h}_{\partial K }^2 
      + \theta \norm[0]{\tilde{l}_p(\bar{q}_h)}_{K}^2 \big)\\
    &\le C \big( \theta \norm[0]{M_K(\tilde{l}_p(\bar{q}_h)) - \tilde{l}_p(\bar{q}_h)}_{K}^2
      + h_K \theta \norm[0]{\tilde{l}_p(\bar{q}_h) - \bar{q}_h}_{\partial K }^2 
      + \theta \norm[0]{\tilde{l}_p(\bar{q}_h)}_{K}^2 \big)\\
    & \le C \big( \theta \norm[0]{\tilde{l}_p(\bar{q}_h)}_K^{2}
      + h_K^2 \xi^{-1} \theta \xi h^{-1}_K \norm[0]{\tilde{l}_p(\bar{q}_h) - \bar{q}_h}_{\partial K}^2  \big).
  \end{align*}
  If $\theta = \gamma \le h^{-2} \xi$, we find
  \begin{equation*}
    \gamma \norm[0]{m_K(\bar{q}_h)}_{K}^2
    \le C \del[1]{ \gamma \norm[0]{\tilde{l}_p(\bar{q}_h)}_K^{2}
    + \xi h^{-1}_K \norm[0]{\tilde{l}_p(\bar{q}_h) - \bar{q}_h}_{\partial K}^2 },
  \end{equation*}
  which is \cref{eq:estimate-m_K}. On the other hand, if
  $\theta = h_K^{-2} \xi \le \gamma$, we find
  \begin{equation*}
    h_K^{-2} \xi \norm[0]{m_K(\bar{q}_h)}_K^2
    \le C \del[1]{ h_K^{-2} \xi \norm[0]{\tilde{l}_p(\bar{q}_h)}_K^2
    + \xi h^{-1}_K \norm[0]{\tilde{l}_p(\bar{q}_h) - \bar{q}_h}_{\partial K}^2 }.
  \end{equation*}
  We again find \cref{eq:estimate-m_K} by noting that
  $h_K^{-2} \xi \le \gamma$ and $ \eta > 1$.
\end{proof}

We are now in a position to prove \cref{lem:mainLemmaStep2Darcy}.

\begin{proof}[Proof of \cref{lem:mainLemmaStep2Darcy}]
  Thanks to the inf-sup condition \cref{eq:inf-sup-H1}, given an
  arbitrary $\bar{q}_h \in \bar{Q}_h^0$, there exists a
  $\tilde{u}_h \in V_h$ such that
  \begin{subequations}
    \begin{align}
      \label{eq:existutildehVh-a}
      \norm[0]{\tilde{u}_h}_{\mathcal{T}_h}
      &\le c \tnorm{(l_p(\bar{q}_h), \bar{q}_h)}_p,
      \\
      \label{eq:existutildehVh-b}
      \tnorm{(l_p(\bar{q}_h), \bar{q}_h)}_p^2
      &= b_h(\tilde{u}_h, (l_p(\bar{q}_h), \bar{q}_h))
      = -(\nabla \cdot \tilde{u}_h, l_p(\bar{q}_h))_{\mathcal{T}_h} + \langle \bar{q}_h, \tilde{u}_h \cdot n \rangle_{\partial \mathcal{T}_h}.
    \end{align}
  \end{subequations}
  In \cref{eq3:localsolver-darcy} choose
  $(v_h,q_h) = (\xi \tilde{u}_h, 0)$ to find, after summing over all
  cells:
  \begin{equation}
    \label{eq:lubarqhtildevh}
    -(l_u(\bar{q}_h),\tilde{u}_h)_{\mathcal{T}_h} + \xi (l_p(\bar{q}_h), \nabla \cdot \tilde{u}_h)_{\mathcal{T}_h}
    = \xi \langle \bar{q}_h, \tilde{u}_h \cdot n \rangle_{\partial\mathcal{T}_h}.
  \end{equation}
  Multiplying \cref{eq:existutildehVh-b} by $\xi$, using
  \cref{eq:lubarqhtildevh}, the Cauchy--Schwarz inequality, and
  \cref{eq:existutildehVh-a}, 
  \begin{equation*}
    \xi \tnorm{(l_p(\bar{q}_h), \bar{q}_h)}_p^2
    = -(l_u(\bar{q}_h),\tilde{u}_h)_{\mathcal{T}_h}
    \le c \tnorm{(l_p(\bar{q}_h), \bar{q}_h)}_p \xi^{1/2} \tnorm{l_u(\bar{q}_h)}_{v,D}.
  \end{equation*}
  We therefore find that
  \begin{equation}
    \label{eq:kappalpqbarboundluqbar}
    \xi^{1/2} \tnorm{(l_p(\bar{q}_h), \bar{q}_h)}_p
    \le c \tnorm{l_u(\bar{q}_h)}_{v,D}.
  \end{equation}
  By definition of $\tnorm{\cdot}_{\boldsymbol{X}_h}$, the norm
  induced by the inner product defined in \cref{eq:darcyInnerProd},
  using \cref{eq:kappalpqbarboundluqbar}, and using
  \cref{lem:previous-condensed-estimate-Darcy,lem:estimate-m_K}
  \begin{equation}
    \label{eq:luqbarlpqbarqbar}
    \begin{split}
      &\tnorm{(l_u(\bar{q}_h), l_p(\bar{q}_h), \bar{q}_h)}_{\boldsymbol{X}_h}^2
      \\
      &\qquad
      \le c \del[1]{ \xi \norm[0]{h_K^{-1/2}(\bar{q}_h - m_K(\bar{q}_h)) }_{\partial \mathcal{T}_h}^2
        + \gamma \norm[0]{\tilde{l}_p(\bar{q}_h)}_{\mathcal{T}_h}^2
        + \xi \eta \norm[0]{h_K^{-1/2}(\tilde{l}_p(\bar{q}_h) - \bar{q}_h)}_{\partial \mathcal{T}_h}^2}
      \\
      &\qquad
      \le c \del[1]{ \xi\tnorm{\bar{q}_h}_{h,p}^2
        + \gamma \norm[0]{\tilde{l}_p(\bar{q}_h)}_{\mathcal{T}_h}^2
        + \xi \tnorm{(\tilde{l}_p(\bar{q}_h), \bar{q}_h)}_p^2
        }.
    \end{split}
  \end{equation}
  Following the same steps as in the proof of \cite[Lemma
  5]{rhebergen2018preconditioning} we furthermore have
  \begin{equation}
    \label{eq:kappaqbarboundltqbarqbar}
    \xi \tnorm{\bar{q}_h}_{h,p}^2 \le c \xi \tnorm{(\tilde{l}_p(\bar{q}_h), \bar{q}_h)}_p^2.
  \end{equation}
  And so, combining
  \cref{eq:luqbarlpqbarqbar,eq:kappaqbarboundltqbarqbar}
  \begin{equation*}
    \tnorm{(l_u(\bar{q}_h), l_p(\bar{q}_h), \bar{q}_h)}_{\boldsymbol{X}_h}^2
    \le c \del[1]{\gamma \norm[0]{\tilde{l}_p(\bar{q}_h)}_{\mathcal{T}_h}^2
      + \xi \tnorm{(\tilde{l}_p(\bar{q}_h), \bar{q}_h)}_p^2
    }
    = c \tnorm{(\tilde{l}_p(\bar{q}_h), \bar{q}_h)}_{q,D}^2.
  \end{equation*}
  The result follows using the coercivity result in
  \cref{eq:coerbndatildeD}.
\end{proof}

\section{The Stokes problem: Proofs}

We recall from \cite[Lemma 4.2]{rhebergen2017analysis} and \cite[Lemma
1]{rhebergen2018preconditioning} that there exist positive uniform constants
$\bar{c}_1,\bar{c}_2$ such that
\begin{subequations}
  \begin{align}
    \label{eq:chcoercivity}
    c_h(\boldsymbol{v}_h, \boldsymbol{v}_h)
    & \ge \bar{c}_1 \tnorm{\boldsymbol{v}_h}_{v,S}^2 \quad \forall \boldsymbol{v}_h \in \boldsymbol{V}_h,
    \\
    \label{eq:bhinfsupStokes}
    \bar{c}_2
    & \le \inf_{\boldsymbol{0}\ne \boldsymbol{q}_h \in \boldsymbol{Q}_h}
      \sup_{\boldsymbol{0}\ne \boldsymbol{v}_h \in \boldsymbol{V}_h}
      \frac{b_h(v_h, \boldsymbol{q}_h)}{\tnorm{\boldsymbol{v}_h}_v \tnorm{\boldsymbol{q}_h}_{0,p}}.    
  \end{align}
\end{subequations}
(The proofs in \cite{rhebergen2017analysis,rhebergen2018preconditioning}
consider norms with the usual gradient. They can be extended to the
case with the symmetric gradient using Korn's inequality
\cite{brenner2004korn}.)

\subsection{Proof of \cref{lem:cbciStokes}}
\label{ss:Proof_cbciStokes}

We first prove \cref{eq:wpconditions_b}. Let
$(\boldsymbol{u}_h,\boldsymbol{p}_h),
(\boldsymbol{v}_h,\boldsymbol{q}_h) \in \boldsymbol{X}_h$. Using the
single-valuedness of $\bar{v}_h$ on interior faces and $\bar{v}_h=0$
on $\Gamma$, the Cauchy--Schwarz inequality, and H\"older's inequality
for sums, we find:
\begin{equation}
  \label{eq:bhboundStokes}
  b_h(v_h, \boldsymbol{q}_{h}) \le c_2 \tnorm{\boldsymbol{q}_{h}}_{q,S} \tnorm{\boldsymbol{v}_h}_{v,S},
\end{equation}
with $c_2>0$ a uniform constant. \Cref{eq:wpconditions_b} follows
after combining \cref{eq:chboundStokes,eq:bhboundStokes}.

We now prove \cref{eq:wpconditions_i}. We proceed in a similar way to
our proof in \cref{ss:Proof_cbciDarcy}. We note that if for all
$(\boldsymbol{u}_h, \boldsymbol{p}_h) \in \boldsymbol{X}_h$ we can
find $(\boldsymbol{v}_h, \boldsymbol{q}_h) \in \boldsymbol{X}_h$
depending on $(\boldsymbol{u}_h, \boldsymbol{p}_h)$ such that
\begin{subequations}
  \begin{align}
  \label{eq:stokes22b-1}
    \tnorm{(\boldsymbol{v}_h,\boldsymbol{q}_h)}_{\boldsymbol{X}_h}
    &\le C_1 \tnorm{(\boldsymbol{u}_h,\boldsymbol{p}_h)}_{\boldsymbol{X}_h},
    \\
    \label{eq:stokes22b-2}
    a_h((\boldsymbol{u}_h, \boldsymbol{p}_h), (\boldsymbol{v}_h, \boldsymbol{q}_h))
    &\ge C_2 \tnorm{(\boldsymbol{u}_h,\boldsymbol{p}_h)}_{\boldsymbol{X}_h}^2,
  \end{align}
\end{subequations}
for positive uniform constants $C_1,C_2$, then
\cref{eq:wpconditions_i} holds with $c_i>0$ a uniform constant. Thanks
to \cref{eq:bhinfsupStokes}, given
$\boldsymbol{p}_h \in \boldsymbol{Q}_h$ there exists a
$\tilde{\boldsymbol{u}}_h \in \boldsymbol{V}_h$ such that
\begin{equation}
  \label{eq:bhvtphutcph}
  b_h(\tilde{u}_h, \boldsymbol{p}_h) = \tnorm{\boldsymbol{p}_h}_{0,p}^2
  \qquad \text{and} \qquad
  \tnorm{\tilde{\boldsymbol{u}}_h}_{v} \le \bar{c}_2 \tnorm{\boldsymbol{p}_h}_{0,p}.
\end{equation}
Define
$\boldsymbol{v}_h := \boldsymbol{u}_h + \delta \nu^{-1}
\tilde{\boldsymbol{u}}_h$ and $\boldsymbol{q}_h := -\boldsymbol{p}_h$
where $\delta > 0$ is a constant that will be determined
below. Similarly to \cref{eq:boundvhvD}, and using $\eta > 1$, we
find
\begin{equation}
  \label{eq:boundvhvSstokes}
  \tnorm{\boldsymbol{v}_h}_{v,S}^2 \le 2 \tnorm{\boldsymbol{u}_h}_{v,S}^2 + 2\delta^2\bar{c}_2^2 \eta \tnorm{\boldsymbol{p}_h}_{q,S}^2.
\end{equation}
\Cref{eq:boundvhvSstokes} together with
$\tnorm{\boldsymbol{q}_h}_{q,S} = \tnorm{\boldsymbol{p}_h}_{q,S}$ implies \cref{eq:stokes22b-1}. Note that
\begin{equation*}
  \begin{split}
    a_h((\boldsymbol{u}_h, \boldsymbol{p}_{h}), (\boldsymbol{v}_h, \boldsymbol{q}_{h}))
    &= c_h(\boldsymbol{u}_h, \boldsymbol{v}_h)
      + b_h(\boldsymbol{v}_h, \boldsymbol{p}_{h})
      + b_h(\boldsymbol{u}_h, \boldsymbol{q}_{h})
    \\
    &= c_h(\boldsymbol{u}_h, \boldsymbol{u}_h)
      + \delta \nu^{-1} c_h(\boldsymbol{u}_h, \tilde{\boldsymbol{u}}_h)
      + \delta \nu^{-1} b_h(\tilde{\boldsymbol{u}}_h, \boldsymbol{p}_{h}).
  \end{split}
\end{equation*}
Using \cref{eq:bhvtphutcph}, \cref{eq:chcoercivity},
\cref{eq:chboundStokes}, Young's inequality, and that $1 > \eta^{-1}$,
we find
\begin{align*}
  a_h((\boldsymbol{u}_h, \boldsymbol{p}_{h}),(\boldsymbol{v}_h, \boldsymbol{q}_{h}))
  &\ge \bar{c}_1 \tnorm{\boldsymbol{u}_h}_{v,S}^2
    - \delta c_1 \nu^{-1/2} \tnorm{\boldsymbol{u}_h}_{v,S} \tnorm{\tilde{\boldsymbol{u}}_h}_v
    + \delta \nu^{-1} \tnorm{\boldsymbol{p}_{h}}_{0,p}^2
  \\
  &\ge \del[2]{\bar{c}_1 - \frac{\delta}{2 \epsilon}} \tnorm{\boldsymbol{u}_h}_{v,S}^2
    + \delta \del[2]{1 - \frac{c_1^2 \epsilon \bar{c}_2}{2}} \tnorm{\boldsymbol{p}_{h}}_{q,S}^2.
\end{align*}
Choosing $\epsilon = 1/(c_1^2\bar{c}_2)$ and
$\delta = \bar{c}_1\epsilon$ we obtain \cref{eq:wpconditions_i}.

\subsection{Proof of \cref{lem:mainLemmaStep2Stokes}}
\label{ss:Proof_mainLemmaStep2Stokes}

Before proving \cref{lem:mainLemmaStep2Stokes} we require the
following preliminary results. First, we recall the following result
from \cite[Lemma 5]{rhebergen2018preconditioning}: There exist
positive uniform constants $c_1, c_2$ such that
\begin{equation}
  \label{eq:Lem5Rheb2018}
  c_1\nu \tnorm{\bar{v}_h}_{h,u}^2
  \le c_h((l_u(\bar{v}_h), \bar{v}_h), (l_u(\bar{v}_h),\bar{v}_h))
  \le c_2 \nu \tnorm{\bar{v}_h}_{h,u}^2
  \quad \forall \bar{v}_h \in \bar{V}_h.
\end{equation}
Next, we prove the following two lemmas.

\begin{lemma}
  \label{lem:condensed-estimate-stokes-lpS}
  Given $(\bar{v}_h, \bar{q}_{h}) \in \bar{V}_h \times \bar{Q}_h$, let
  $l_u(\bar{v}_h,\bar{q}_h)$ and $l_p(\bar{v}_h, \bar{q}_h)$ be as
  defined in \cref{lem:weak-reduced-stokes}. There exists a positive
  uniform constant $\tilde{c}_1$ such that
  \begin{equation*}
    \nu^{-1} \norm[0]{l_{p}(\bar{v}_h, \bar{q}_{h})}_{\mathcal{T}_h}^2
    \le \tilde{c}_1 \eta \del[2]{
      \tnorm{(l_u(\bar{v}_h, \bar{q}_{h}), \bar{v}_h)}_{v,S}^2
      + \nu^{-1} \eta^{-1} \norm[0]{h_K^{1/2} \bar{q}_{h}}_{\partial \mathcal{T}_h}^2
    }.
  \end{equation*}
\end{lemma}
\begin{proof}
  For ease of notation in this proof we define
  $l_u := l_u(\bar{v}_h, \bar{q}_{h})$ and
  $l_{p} := l_{p}(\bar{v}_h, \bar{q}_{h})$. Thanks to
  \cref{eq:bhinfsupStokes}, given
  $(l_{p}, \eta^{-1/2} \bar{q}_{h}) \in \boldsymbol{Q}_h$ there exists
  a
  $\tilde{\boldsymbol{v}}_h := (\tilde{v}_h, \tilde{\bar{v}}_h) \in
  \boldsymbol{V}_h$, such that
  \begin{equation}
    \label{eq:aux1-condensed-estimate-stokes-lpS}
    b_h(\tilde{v}_h, (l_{p}, \eta^{-1/2} \bar{q}_{h}))
    = \tnorm{(l_{p}, \eta^{-1/2} \bar{q}_{h})}_{0,p}^2
    \qquad \text{and} \qquad
    \tnorm{\tilde{\boldsymbol{v}}_h}_v \le \bar{c}_2 \tnorm{(l_{p}, \eta^{-1/2} \bar{q}_{h})}_{0,p}.
  \end{equation}
  Then, choosing
  $v_h = \nu^{-1} (\tilde{v}_h - m_K(\tilde{\bar{v}}_h))$ $q_{h} = 0$,
  and $s=0$ in \cref{eq3:localsolver-stokes}, and adding over all the
  elements, we find:
  \begin{equation}
      \label{eq:epsluepstildev}
      \begin{split}
          &(\varepsilon (l_u), \varepsilon (\tilde{v}_h))_{\mathcal{T}_h}
      - \langle \varepsilon (l_u) n, \tilde{v}_h - m_K(\tilde{\bar{v}}_h)\rangle_{\partial \mathcal{T}_h}
      - \langle l_u, \varepsilon ( \tilde{v}_h) n \rangle_{\partial \mathcal{T}_h}
    \\
    &+ \eta \langle h_K^{-1} l_u, \tilde{v}_h - m_K(\tilde{\bar{v}}_h)\rangle_{\partial \mathcal{T}_h}
      - \nu^{-1} (\nabla \cdot \tilde{v}_h, l_{p})_{\mathcal{T}_h}
    \\
    &\quad=
      - \langle \varepsilon(\tilde{v}_h)n, \bar{v}_h \rangle_{\partial \mathcal{T}_h}
      + \eta \langle h_K^{-1}\bar{v}_h, \tilde{v}_h - m_K(\tilde{\bar{v}}_h)\rangle_{\partial \mathcal{T}_h}
      - \nu^{-1} \langle \bar{q}_{h}, (\tilde{v}_h - m_K(\tilde{\bar{v}}_h)) \cdot n \rangle_{\partial \mathcal{T}_h}.
      \end{split}
  \end{equation}
  From \cref{eq:aux1-condensed-estimate-stokes-lpS} we have
  $b_h(\tilde{v}_h,(l_p,0)) = -(l_p,\nabla \cdot
  \tilde{v}_h)_{\mathcal{T}_h} = \tnorm{(l_p,0)}_{0,p}^2 =
  \norm[0]{l_p}_{\mathcal{T}_h}^2$. Since
  $\nabla \cdot \tilde{v}_h \in Q_h$ this implies
  $-\nabla \cdot \tilde{v}_h = l_p$.  Using this and rearranging terms
  in \cref{eq:epsluepstildev}, applying the Cauchy--Schwarz
  inequality, a discrete trace inequality \cite[Lemma
  1.46]{di2011mathematical} with constant $c_t > 0$, using that
  $\eta > 1$, and Young's inequality, we find
  \begin{equation}
    \label{eq:numin1lpbound}
    \begin{split}
      &\nu^{-1} \norm[0]{l_{p}}_{\mathcal{T}_h}^2
      \\
      &\le \big(\frac{1}{2 \epsilon_1}
      + \frac{c_t^2}{2 \epsilon_2}
      + \frac{c_t^2}{2 \epsilon_3}\big) \nu \norm[0]{\varepsilon( l_u)}_{\mathcal{T}_h}^2
      + \big(\frac{c_t^2}{2 \epsilon_4}
      + \frac{1}{2 \epsilon_5}
      + \frac{1}{2 \epsilon_6}\big) \nu \eta
      \norm[0]{h_K^{-1/2}(l_u - \bar{v}_h)}_{\partial \mathcal{T}_h}^2\\
      &\quad + \big(\frac{\epsilon_1}{2}
      + \frac{\epsilon_4}{2}\big) \nu^{-1} \norm[0]{\varepsilon ( \tilde{v}_h)}_{\mathcal{T}_h}^2
      + \big( \frac{\epsilon_2}{2} + \frac{\epsilon_5}{2}
      + \frac{\epsilon_7}{2} \big)\nu^{-1} \eta \norm[0]{h_K^{-1/2}( \tilde{v}_h - \tilde{\bar{v}}_h)}_{\partial \mathcal{T}_h}^2\\
      &\quad + \big(\frac{\epsilon_3}{2}
      + \frac{\epsilon_6}{2} + \frac{\epsilon_8}{2} \big)
      \nu^{-1} \eta \norm[0]{h_K^{-1/2} (\tilde{\bar{v}}_h - m_K(\tilde{\bar{v}}_h)}_{\partial \mathcal{T}_h}^2
      + \big( \frac{1}{2 \epsilon_7} 
      + \frac{1}{2 \epsilon_8} \big)
      \nu^{-1} \eta^{-1} \norm[0]{h_K^{1/2} \bar{q}_{h}}_{\partial \mathcal{T}_h}^2,      
    \end{split}
  \end{equation}
  where $\epsilon_i > 0$, $i \in \cbr[0]{1, \hdots, 8}$ are constants
  that will be chosen below. We remark that there exists a uniform
  constant $C>0$ such that
  (cf. \cite[eq. (47)]{rhebergen2018preconditioning})
  \begin{equation*}
    \nu^{-1} \eta \norm[0]{h_K^{-1/2} (\tilde{\bar{v}}_h -
      m_K(\tilde{\bar{v}}_h))}_{\partial \mathcal{T}_h}^2
    \le C \nu^{-1} \eta \tnorm{\tilde{\boldsymbol{v}}_h}_v^2.
  \end{equation*}
  The result now follows by combining this estimate and
  \cref{eq:aux1-condensed-estimate-stokes-lpS} with
  \cref{eq:numin1lpbound} and choosing
  $\epsilon_1 = \epsilon_2 = \epsilon_4 = \epsilon_5 = \epsilon_7 =
  1/(8\bar{c}_2)$ and
  $\epsilon_3 = \epsilon_6 = \epsilon_8 = 1/(8C\bar{c}_2\eta)$.
\end{proof}

\begin{lemma}
  \label{lem:condensed-estimate-stokes}
  Given $(\bar{v}_h, \bar{q}_{h}) \in \bar{V}_h \times \bar{Q}_h$, let
  $l_u(\bar{v}_h,\bar{q}_h)$ and $l_p(\bar{v}_h, \bar{q}_h)$ be as
  defined in \cref{lem:weak-reduced-stokes}. There exists a positive
  uniform constant $\tilde{c}_2$ such that
  \begin{equation*}
    \tnorm{(l_u(\bar{v}_h, \bar{q}_{h}), \bar{v}_h)}_{v,S}^2
    \le \tilde{c}_2 \del[2]{\nu \eta \tnorm{\bar{v}_h}_{h,u}^2
      + \nu^{-1} \eta^{-1} \norm[0]{h_K^{1/2}\bar{q}_{h}}_{\partial \mathcal{T}_h}^2}.
  \end{equation*}
\end{lemma}
\begin{proof}
  For ease of notation in this proof we define
  $l_u := l_u(\bar{v}_h, \bar{q}_{h})$ and
  $l_{p} := l_{p}(\bar{v}_h, \bar{q}_{h})$. Choosing
  $v_h = l_u - m_K(\bar{v}_h)$ and $q_{h} = - l_{p}$ in
  \cref{eq3:localsolver-stokes}, and reordering, we find:
  \begin{align*}
    &\nu \norm[0]{\varepsilon (l_u)}_K^2
      + 2 \nu \langle \varepsilon(l_u)n, \bar{v}_h - l_u) \rangle_{\partial K}
      + \nu \eta h_K^{-1} \norm[0]{l_u - \bar{v}_h}_{\partial K}^2
    \\
    &\quad = \nu \langle\varepsilon(l_u)n, \bar{v}_h - m_K(\bar{v}_h)\rangle_{\partial K}
      + \nu \eta h_K^{-1} \langle \bar{v}_h - m_K(\bar{v}_h), \bar{v}_h - l_u\rangle_{\partial K}
    \\
    &\qquad - \langle \bar{q}_{h}, (l_u - \bar{v}_h) \cdot n\rangle_{\partial K}
      - \langle \bar{q}_{h}, (\bar{v}_h - m_K(\bar{v}_h)) \cdot n\rangle_{\partial K}.
  \end{align*}
  By the Cauchy--Schwarz, Young's, and discrete trace
  \cite[Lemma 1.46]{di2011mathematical} inequalities,
  \begin{align*}
    &\nu \norm[0]{\varepsilon(l_u)}_K^2
      + 2 \nu \langle \varepsilon(l_u)n, \bar{v}_h - l_u\rangle_{\partial K}
      + \nu \eta h_K^{-1} \norm[0]{l_u - \bar{v}_h}^2_{\partial K}
    \\
    &\quad \le \frac{\epsilon_1 \nu }{2} \norm[0]{\varepsilon(l_u)}_{K}^2
      + \big( \frac{1}{2 \epsilon_2}
      + \frac{1}{2 \epsilon_3} \big)
      \nu \eta h_K^{-1} \norm[0]{l_u - \bar{v}_h}_{\partial K}^2
      \\
    &\qquad+ \big( \frac{c_t^2}{2 \epsilon_1}
      + \frac{\epsilon_2}{2} + \frac{1}{2}\big)
      \nu \eta h_K^{-1} \norm[0]{\bar{v}_h - m_K(\bar{v}_h)}_{\partial K}^2
      + \big( \frac{\epsilon_3}{2}
      + \frac{1}{2} \big)
      \nu^{-1} \eta^{-1} h_K \norm[0]{\bar{q}_{h}}_{\partial K}^2,
  \end{align*}
  where $\epsilon_i > 0$, $i \in \cbr[0]{1,2,3}$ are the Young's
  inequality constants that 
  will be chosen below. Summing over all
  cells and applying \cref{eq:chcoercivity}, we find:
  \begin{multline*}
    \bar{c}_1 \tnorm{(l_u, \bar{v}_h)}_{v,S}^2
    \le \frac{\epsilon_1 \nu }{2} \norm[0]{\varepsilon (l_u)}_{\mathcal{T}_h}^2
      + \big( \frac{1}{2 \epsilon_2}
      + \frac{1}{2 \epsilon_3} \big)
      \nu \eta \norm[0]{h_K^{-1/2}(l_u - \bar{v}_h)}_{\partial \mathcal{T}_h}^2
    \\
    + \big( \frac{c_t^2}{2 \epsilon_1}
      + \frac{\epsilon_2}{2} + \frac{1}{2}\big)
      \nu \eta \tnorm{\bar{v}_h}_{h,u}^2
      + \big( \frac{\epsilon_3}{2}
      + \frac{1}{2} \big)
      \nu^{-1} \eta^{-1} \norm[0]{h_K^{1/2}\bar{q}_{h}}_{\partial \mathcal{T}_h}^2.
  \end{multline*}
  The result follows after choosing $\epsilon_1 = \bar{c}_1$ and
  $\epsilon_2 = \epsilon_3 = \frac{2}{\bar{c}_1}$.
\end{proof}

We now prove \cref{lem:mainLemmaStep2Stokes}.

\begin{proof}[Proof of \cref{lem:mainLemmaStep2Stokes}]
  By definition of the norm $\tnorm{\cdot}_{\boldsymbol{X}_h}$, we
  have:
  \begin{multline}
    \label{eq:apluvlpvXhnormwriteout}
    \tnorm{(l_u(\bar{v}_h,\bar{q}_h), \bar{v}_h, l_p(\bar{v}_h,\bar{q}_h), \bar{v}_h)}_{\boldsymbol{X}_h}^2
    =
    \tnorm{(l_u(\bar{v}_h, \bar{q}_{h}), \bar{v}_h)}_{v,S}^2
    \\
    +    
    \nu^{-1} \norm[0]{l_{p}(\bar{v}_h, \bar{q}_{h})}_{\mathcal{T}_h}^2
    +
    \nu^{-1} \eta^{-1} \norm[0]{h_K^{1/2} \bar{q}_{h}}_{\partial \mathcal{T}_h}^2.
  \end{multline}
  The result follows after combining \cref{eq:apluvlpvXhnormwriteout}
  with
  \cref{lem:condensed-estimate-stokes-lpS,lem:condensed-estimate-stokes},
  and \cref{eq:Lem5Rheb2018}, and using that $\eta > 1$.
\end{proof}

\end{document}